\RequirePackage{filecontents}

\documentclass[reqno, 11pt, a4paper]{amsart}
\usepackage{amsmath}
\usepackage{enumerate}

\usepackage{color}
\usepackage{ifpdf}
  \usepackage{graphicx}
\ifpdf 
  \DeclareGraphicsExtensions{.pdf,.png,.mps}
  \usepackage{pgf}
  \usepackage{epic}
\else 
  \usepackage{graphicx}
  \DeclareGraphicsExtensions{.eps,.bmp}
  \usepackage{pgf}
\fi
\usepackage{bez123}
\usepackage{wrapfig}

\newtheorem{thm}{Theorem}
\newtheorem{cor}[thm]{Corollary}
\newtheorem{lem}[thm]{Lemma}

\newtheorem{innercustomthm}{Theorem}
\newenvironment{cthm}[1]
  {\renewcommand\theinnercustomthm{#1}\innercustomthm}
  {\endinnercustomthm}
\theoremstyle{definition}

\newcommand{\T}{\mathcal{T}}

\newcommand{\TV}{\mathcal{TV}}

\DeclareMathOperator{\outdeg}{outdeg}
\DeclareMathOperator{\lev}{lev}
\DeclareMathOperator{\eld}{eld}
\DeclareMathOperator{\Cat}{Cat}

\begin{document}
\author{Sen-Peng Eu}
\address[Sen-Peng Eu]{Department of Mathematics, National Taiwan Normal University, Taipei 116, Taiwan, ROC}
\email{speu@math.ntnu.edu.tw}

\author{Seunghyun Seo}
\address[Seunghyun Seo]{Department of Mathematics Education, Kangwon National University, Chuncheon 24341, Korea}
\email{shyunseo@kangwon.ac.kr}

\author{Heesung Shin$^\dag$}
\address[Heesung Shin]{Department of Mathematics, Inha University, Incheon 22212, Korea}
\email{shin@inha.ac.kr}
\thanks{\dag Corresponding author}

\title[Enumerations of vertices among all rooted ordered trees]%
{Enumerations of vertices among all rooted ordered trees with levels and degrees}
\date{\today}

\begin{abstract}
In this paper we enumerate and give bijections for the following four sets of vertices among rooted ordered trees of a fixed size: (i) first-children of degree $k$ at level $\ell$, (ii) non-first-children of degree $k$ at level $\ell-1$, (iii) leaves having $k-1$ elder siblings at level $\ell$, and (iv) non-leaves of outdegree $k$ at level $\ell-1$. Our results unite and generalize several previous works in the literature.
\end{abstract}

\maketitle

\section{Introduction}
\label{sec:intro}
Let $\mathcal{T}_n$ be the set of \emph{rooted ordered trees} with $n$
edges. It is well known that the cardinality of $\mathcal{T}_n$ is
the \emph{$n$-th Catalan number}
$$\Cat_n = \frac{1}{n+1}\binom{2n}{n}.$$
For example, there are 5 rooted ordered trees with $3$ edges, see
Figure~\ref{fig:tree}.
\begin{figure}[t]
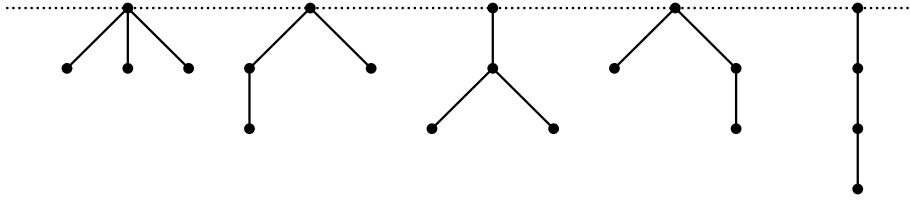

\centering
\begin{pgfpicture}{-2.00mm}{13.44mm}{122.00mm}{42.56mm}
\pgfsetxvec{\pgfpoint{0.80mm}{0mm}}
\pgfsetyvec{\pgfpoint{0mm}{0.80mm}}
\color[rgb]{0,0,0}\pgfsetlinewidth{0.30mm}\pgfsetdash{}{0mm}
\pgfmoveto{\pgfxy(20.00,50.00)}\pgflineto{\pgfxy(10.00,40.00)}\pgfstroke
\pgfmoveto{\pgfxy(20.00,50.00)}\pgflineto{\pgfxy(20.00,40.00)}\pgfstroke
\pgfmoveto{\pgfxy(20.00,50.00)}\pgflineto{\pgfxy(30.00,40.00)}\pgfstroke
\pgfmoveto{\pgfxy(50.00,50.00)}\pgflineto{\pgfxy(40.00,40.00)}\pgfstroke
\pgfmoveto{\pgfxy(50.00,50.00)}\pgflineto{\pgfxy(60.00,40.00)}\pgfstroke
\pgfmoveto{\pgfxy(40.00,40.00)}\pgflineto{\pgfxy(40.00,30.00)}\pgfstroke
\pgfmoveto{\pgfxy(80.00,50.00)}\pgflineto{\pgfxy(80.00,40.00)}\pgfstroke
\pgfmoveto{\pgfxy(80.00,40.00)}\pgflineto{\pgfxy(70.00,30.00)}\pgfstroke
\pgfmoveto{\pgfxy(80.00,40.00)}\pgflineto{\pgfxy(90.00,30.00)}\pgfstroke
\pgfmoveto{\pgfxy(110.00,50.00)}\pgflineto{\pgfxy(100.00,40.00)}\pgfstroke
\pgfmoveto{\pgfxy(110.00,50.00)}\pgflineto{\pgfxy(120.00,40.00)}\pgfstroke
\pgfmoveto{\pgfxy(120.00,40.00)}\pgflineto{\pgfxy(120.00,30.00)}\pgfstroke
\pgfmoveto{\pgfxy(140.00,50.00)}\pgflineto{\pgfxy(140.00,40.00)}\pgfstroke
\pgfmoveto{\pgfxy(140.00,40.00)}\pgflineto{\pgfxy(140.00,30.00)}\pgfstroke
\pgfmoveto{\pgfxy(140.00,30.00)}\pgflineto{\pgfxy(140.00,20.00)}\pgfstroke
\pgfsetdash{{0.30mm}{0.50mm}}{0mm}\pgfmoveto{\pgfxy(0.00,50.00)}\pgflineto{\pgfxy(150.00,50.00)}\pgfstroke
\pgfcircle[fill]{\pgfxy(140.00,20.00)}{0.56mm}
\pgfsetdash{}{0mm}\pgfcircle[stroke]{\pgfxy(140.00,20.00)}{0.56mm}
\pgfcircle[fill]{\pgfxy(140.00,30.00)}{0.56mm}
\pgfcircle[stroke]{\pgfxy(140.00,30.00)}{0.56mm}
\pgfcircle[fill]{\pgfxy(140.00,40.00)}{0.56mm}
\pgfcircle[stroke]{\pgfxy(140.00,40.00)}{0.56mm}
\pgfcircle[fill]{\pgfxy(140.00,50.00)}{0.56mm}
\pgfcircle[stroke]{\pgfxy(140.00,50.00)}{0.56mm}
\pgfcircle[fill]{\pgfxy(120.00,30.00)}{0.56mm}
\pgfcircle[stroke]{\pgfxy(120.00,30.00)}{0.56mm}
\pgfcircle[fill]{\pgfxy(120.00,40.00)}{0.56mm}
\pgfcircle[stroke]{\pgfxy(120.00,40.00)}{0.56mm}
\pgfcircle[fill]{\pgfxy(110.00,50.00)}{0.56mm}
\pgfcircle[stroke]{\pgfxy(110.00,50.00)}{0.56mm}
\pgfcircle[fill]{\pgfxy(20.00,40.00)}{0.56mm}
\pgfcircle[stroke]{\pgfxy(20.00,40.00)}{0.56mm}
\pgfcircle[fill]{\pgfxy(10.00,40.00)}{0.56mm}
\pgfcircle[stroke]{\pgfxy(10.00,40.00)}{0.56mm}
\pgfcircle[fill]{\pgfxy(20.00,50.00)}{0.56mm}
\pgfcircle[stroke]{\pgfxy(20.00,50.00)}{0.56mm}
\pgfcircle[fill]{\pgfxy(100.00,40.00)}{0.56mm}
\pgfcircle[stroke]{\pgfxy(100.00,40.00)}{0.56mm}
\pgfcircle[fill]{\pgfxy(90.00,30.00)}{0.56mm}
\pgfcircle[stroke]{\pgfxy(90.00,30.00)}{0.56mm}
\pgfcircle[fill]{\pgfxy(70.00,30.00)}{0.56mm}
\pgfcircle[stroke]{\pgfxy(70.00,30.00)}{0.56mm}
\pgfcircle[fill]{\pgfxy(80.00,40.00)}{0.56mm}
\pgfcircle[stroke]{\pgfxy(80.00,40.00)}{0.56mm}
\pgfcircle[fill]{\pgfxy(80.00,50.00)}{0.56mm}
\pgfcircle[stroke]{\pgfxy(80.00,50.00)}{0.56mm}
\pgfcircle[fill]{\pgfxy(60.00,40.00)}{0.56mm}
\pgfcircle[stroke]{\pgfxy(60.00,40.00)}{0.56mm}
\pgfcircle[fill]{\pgfxy(40.00,30.00)}{0.56mm}
\pgfcircle[stroke]{\pgfxy(40.00,30.00)}{0.56mm}
\pgfcircle[fill]{\pgfxy(40.00,40.00)}{0.56mm}
\pgfcircle[stroke]{\pgfxy(40.00,40.00)}{0.56mm}
\pgfcircle[fill]{\pgfxy(50.00,50.00)}{0.56mm}
\pgfcircle[stroke]{\pgfxy(50.00,50.00)}{0.56mm}
\pgfcircle[fill]{\pgfxy(30.00,40.00)}{0.56mm}
\pgfcircle[stroke]{\pgfxy(30.00,40.00)}{0.56mm}
\end{pgfpicture}%
\caption{The five trees in $\mathcal{T}_3$}
\label{fig:tree}
\end{figure}
Clearly the number of vertices among trees in $\mathcal{T}_n$ is
$$(n+1) \Cat_n = \binom{2n}{n}.$$

Given a rooted ordered tree,
a vertex $v$ is a \emph{child} of a vertex $u$ and $u$ is the \emph{parent} of $v$
if $v$ is directly connected to $u$ when moving away from the root.
A vertex without children is called a \emph{leaf}.
Note that by definition the root is not a child.
Vertices with the same parent are called \emph{siblings}.
Since siblings are linearly ordered, when drawing trees
the siblings are put in the left-to-right order.
Siblings to the left of $v$ are called an \emph{elder} siblings of $v$.
The leftmost vertex among siblings is called the \emph{first-child}.
In Figure~\ref{fig:tree}, there are 10 first-children as well as 10 leaves, 
which is precisely a half of all 20 vertices in trees in $\mathcal{T}_3$.

In 1999, Shapiro \cite{Sha99} proved the following using generating functions.
\begin{thm}
\label{thm:base}
For any positive integer $n$, the following four sets of vertices among trees in $\T_n$ are equinumerous:
\begin{enumerate}[(i)]
\item first-children, \label{item:1}
\item non-first-children, \label{item:2}
\item leaves, and \label{item:3}
\item non-leaves. \label{item:4}
\end{enumerate}
Here, the cardinality of each set is $$\frac{1}{2} {2n \choose n},$$ which is a half of the number of vertices among trees in $\T_n$.
\end{thm}
Seo and Shin \cite{SS02} gave an involution proving \eqref{item:3} and \eqref{item:4} are equinumerous.
The equality with the other two sets is easily seen:
\eqref{item:1} and \eqref{item:4} are equinumerous since each non-leaf has its unique first-child
and the union of \eqref{item:1} and \eqref{item:2} are the same with \eqref{item:3} and \eqref{item:4}.

\subsection{Degree and outdegree}
\label{subsec:degree}

The \emph{degree} of a vertex is the number of edges incident to it.
As every edge has a natural outward direction away from the root,
we can have the notion of the \emph{outdegree} of a vertex $v$,
which is the number of edges starting at $v$ and pointing away from the root.
Since each vertex has a degree and each non-leaf has a positive
outdegree, Theorem~\ref{thm:base} can be restated as follows.

\begin{cthm}{\ref{thm:base}'}[Shapiro, 1999]
Let $n\ge 1$. 
Among trees in $\mathcal{T}_n$, the number of vertices of positive degree
equals twice the number of vertices of positive outdegree.
\end{cthm}

In 2004, Eu, Liu, and Yeh \cite{ELY04} proved combinatorially the
following by constructing a two-to-one correspondence which answered
a problem posed by Deutsch and Shapiro \cite[p.\,259]{DS01}.

\begin{thm}[Eu, Liu, and Yeh, 2004]
\label{thm:ELY}
Let $n\ge 1$. Among trees in $\mathcal{T}_n$, the number of vertices of odd degree equals
twice the number of vertices of odd positive outdegree.
\end{thm}

In light of these two results, it is natural to ask if more can be
said.
In Corollary~\ref{thm:ESS1} 
we will prove that for any positive integer $k$
the number of vertices of degree $k$ always equals \emph{twice} the number of that of outdegree $k$
among trees in $\mathcal{T}_n$.  

\subsection{Even and odd level}
\label{subsec:level}

A vertex $v$ in a rooted tree $T$ is at \emph{level} $\ell$ if the
distance (number of edges) from the root to $v$ is $\ell$.
By an involution, the following result was obtained by Chen, Li, and Shapiro \cite{CLS07}.

\begin{thm}[Chen, Li, and Shapiro, 2007]
\label{thm:CLS}
The number of vertices at odd levels equals the number of vertices at even levels among trees in $\mathcal{T}_n$.
\end{thm}

It is again natural to ask if more can be said.
In Corollary~\ref{thm:ESS} of this paper we will prove that for any positive integer $k$
the number of vertices of degree $k$ at odd
levels always equals the number of vertices of degree $k$ at even
levels among trees in $\mathcal{T}_n$.

\subsection{Main result}
In this paper we generalize the above by regarding both
degrees and levels simultaneously. We will give a combinatorial
proof for the following main result.

\begin{thm}
\label{thm:main}
Given $n \ge 1$,
for any two positive integers $k$ and $\ell$,
there are one-to-one correspondences between the following four sets of vertices among trees in $\T_n$:
\begin{enumerate}[(i)]
\item \label{item:c1} first-children of degree $k$ at level $\ell$,
\item \label{item:c2} non-first-children of degree $k$ at level $(\ell-1)$,
\item \label{item:c3} leaves having exactly $(k-1)$ elder siblings at level $\ell$, and
\item \label{item:c4} non-leaves of outdegree $k$ at level $(\ell-1)$.
\end{enumerate}
Here, the cardinality of each set is
\begin{align}
\label{eq:main}
\frac{k+2\ell-2}{2n-k} { 2n -k \choose n +\ell -1}.
\end{align}
\end{thm}

The rest of the paper is organized as follows.
In Section~\ref{sec:coro} we derive several corollaries from Theorem~\ref{thm:main}, including generalizations of Theorem~\ref{thm:base} through \ref{thm:CLS}.
The proof of Theorem~\ref{thm:main} is given in the next two sections:
In Section~\ref{sec:proof} we construct bijections between these four sets
while in Section~\ref{sec:enumeration} we show combinatorially that the cardinality is given by \eqref{eq:main}.

\section{Corollaries}
\label{sec:coro}
In 2012, Cheon and Shapiro \cite[Example 2.2]{CS12} gave a formula for the number of vertices of outdegree $k$ by generating function arguments.
Summing over $\ell \ge 1$ in Theorem~\ref{thm:main} yields the following,
in which the fourth item recovers bijectively the result of Cheon and Shapiro.
\begin{cor}
\label{cor:degree}
Given $n \ge 1$ and $k \ge 1$,
there are one-to-one correspondences between the following four sets of vertices among trees in $\T_n$:
\begin{enumerate}[(i)]
\item \label{item:d1} first-children of degree $k$,
\item \label{item:d2} non-first-children of degree $k$,
\item \label{item:d3} leaves having exactly $(k-1)$ elder siblings, and
\item \label{item:d4} non-leaves of outdegree $k$.
\end{enumerate}
Here, the cardinality of each set is
$$\binom{2n-k-1}{n-1}.$$
\end{cor}

\begin{proof}
The result follows by a telescoping summation on $\ell \ge 1$ using the formula
\begin{align*}
\frac{k+2\ell-2}{2n-k} \binom{ 2n -k}{n +\ell -1} = &
\binom{2n -k -1}{n + \ell -2} - \binom{2n -k -1}{n + \ell -1}.
\end{align*}
\end{proof}

The result mentioned in Subsection~\ref{subsec:degree}
can be obtained from \eqref{item:d1}, \eqref{item:d2}, and \eqref{item:d4} of Corollary~\ref{cor:degree}.

\begin{cor}\label{thm:ESS1}
Given $n \ge 1$ and $k \ge 1$,
the number of vertices of degree $k$ equals twice the number of vertices of outdegree $k$
among trees in $\mathcal{T}_n$.
\end{cor}

Note that Corollary~\ref{thm:ESS1} is a refinement of Theorem~\ref{thm:ELY}.
Summing over all $k$ in Theorem~\ref{thm:main} yields the following.
\begin{cor}
Given $n \ge 1$ and $\ell \ge 1$,
there are one-to-one correspondences between the following four sets of vertices among trees in $\T_n$:
\begin{enumerate}[(i)]
\item first-children at level $\ell$,
\item non-first-children at level $(\ell-1)$,
\item leaves at level $\ell$, and
\item non-leaves at level $(\ell-1)$.
\end{enumerate}
Here, the cardinality of each set is
$$\frac{\ell}{n} \binom{2n}{n+\ell}.$$
\end{cor}

\begin{proof}
The result follows by a telescoping summation on $k \ge 1$ using the following formula,
\begin{align*}
\frac{k+2\ell-2}{2n-k} \binom{ 2n -k }{ n + \ell - 1}
= \frac{k + 2\ell -1 }{2n-k+1} \binom{2n - k + 1}{n + \ell}
- \frac{k + 2\ell    }{2n-k  } \binom{2n - k    }{n + \ell}.
\end{align*}
\end{proof}

The result mentioned in Subsection~\ref{subsec:level} can be obtained from (\ref{item:c1}) and (\ref{item:c2}) in Theorem~\ref{thm:main}.
\begin{cor}
\label{thm:ESS}
Given $n \ge 1$,
the number of vertices of degree $k$ at odd levels equals the number of vertices of degree $k$ at even levels among trees in $\T_n$.
\end{cor}


\section{Proof of Theorem~\ref{thm:main}: Bijections}
\label{sec:proof}
Denote $\mathcal{TV}_n$ by the set of pairs $(T,v)$ satisfying $T \in \mathcal{T}_n$ and $v \in V(T)$.
Given positive integers $n$, $k$, and $\ell$, we let
\begin{enumerate}[(i)]
\item \label{item:e1} $\mathcal{A}$ denote the set of $(T,v) \in \mathcal{TV}_n$ such that $v$ is a first-child of degree $k$ at level $\ell$ in $T$,
\item \label{item:e2} $\mathcal{B}$ denote the set of $(T,v) \in \mathcal{TV}_n$ such that $v$ is a non-first-child of degree $k$ at level $(\ell-1)$ in $T$,
\item \label{item:e3} $\mathcal{C}$ denote the set of $(T,v) \in \mathcal{TV}_n$ such that $v$ is a leaf having exactly $(k-1)$ elder siblings at level $\ell$ in $T$, and
\item \label{item:e4} $\mathcal{D}$ denote the set of $(T,v) \in \mathcal{TV}_n$ such that $v$ is a non-leaf of outdegree $k$ at level $(\ell-1)$ in $T$.
\end{enumerate}


\subsection*{(\ref{item:e1}) $\Leftrightarrow$ (\ref{item:e4}).}
A bijection from $\mathcal{A}$ to $\mathcal{D}$ is constructed as follows:
given $(T,v) \in \mathcal{A}$, find the parent vertex $u$ of $v$.
As the Figure~\ref{fig:1to4}, consider a subtree $D_v$ of $v$ and another subtree $D_u$ of $u$ on the right of the edge $(u,v)$.
By interchanging $D_v$ and $D_u$ we get the new tree $T'$ with the vertex $u$ such that
\begin{align*}
\outdeg(T',u) &= \deg(T,v) = k, & \lev(T',u) &= \lev(T,v) -1 = \ell -1,
\end{align*}
where $\lev(T,v)$ (resp. $\deg(T,v)$, $\outdeg(T,v)$) means the level (resp. degree, outdegree) of a vertex $v$ in a tree $T$.
Thus, $(T',u)$ belongs to $\mathcal{D}$. Since this interchanging action is reversible, it is a one-to-one correspondence.

\begin{figure}[t]
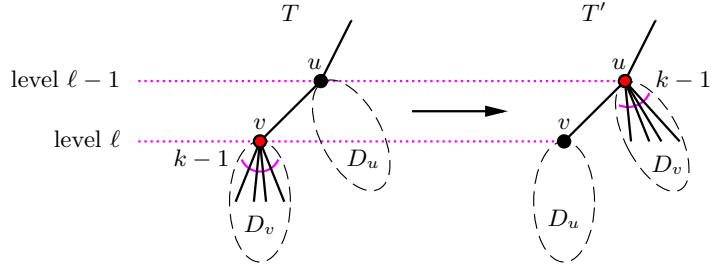

\centering
\begin{pgfpicture}{-14.98mm}{30.00mm}{74.42mm}{68.51mm}
\pgfsetxvec{\pgfpoint{0.80mm}{0mm}}
\pgfsetyvec{\pgfpoint{0mm}{0.80mm}}
\color[rgb]{0,0,0}\pgfsetlinewidth{0.30mm}\pgfsetdash{}{0mm}
\color[rgb]{1,0,1}\pgfmoveto{\pgfxy(79.00,66.00)}\pgfcurveto{\pgfxy(79.96,65.59)}{\pgfxy(81.04,65.59)}{\pgfxy(82.00,66.00)}\pgfcurveto{\pgfxy(82.90,66.38)}{\pgfxy(83.62,67.10)}{\pgfxy(84.00,68.00)}\pgfstroke
\pgfmoveto{\pgfxy(17.00,57.00)}\pgfcurveto{\pgfxy(17.50,55.79)}{\pgfxy(18.69,55.00)}{\pgfxy(20.00,55.00)}\pgfcurveto{\pgfxy(21.31,55.00)}{\pgfxy(22.50,55.79)}{\pgfxy(23.00,57.00)}\pgfstroke
\pgfsetdash{{0.30mm}{0.50mm}}{0mm}\pgfmoveto{\pgfxy(0.00,60.00)}\pgflineto{\pgfxy(69.00,60.00)}\pgfstroke
\pgfmoveto{\pgfxy(0.00,70.00)}\pgflineto{\pgfxy(79.00,70.00)}\pgfstroke
\color[rgb]{0,0,0}\pgfsetdash{}{0mm}\pgfmoveto{\pgfxy(35.00,80.00)}\pgflineto{\pgfxy(35.00,80.00)}\pgfstroke
\pgfmoveto{\pgfxy(35.00,80.00)}\pgflineto{\pgfxy(30.00,70.00)}\pgfstroke
\pgfmoveto{\pgfxy(30.00,70.00)}\pgflineto{\pgfxy(20.00,60.00)}\pgfstroke
\pgfmoveto{\pgfxy(20.00,60.00)}\pgflineto{\pgfxy(16.00,50.00)}\pgfstroke
\pgfmoveto{\pgfxy(20.00,60.00)}\pgflineto{\pgfxy(19.00,50.00)}\pgfstroke
\pgfmoveto{\pgfxy(20.00,60.00)}\pgflineto{\pgfxy(24.00,50.00)}\pgfstroke
\pgfsetdash{{2.00mm}{1.00mm}}{0mm}\pgfsetlinewidth{0.15mm}\pgfellipse[stroke]{\pgfxy(20.00,50.00)}{\pgfxy(5.00,0.00)}{\pgfxy(0.00,10.00)}
\pgfellipse[stroke]{\pgfxy(35.00,61.00)}{\pgfxy(-4.44,-2.29)}{\pgfxy(4.58,-8.89)}
\pgfcircle[fill]{\pgfxy(30.00,70.00)}{0.80mm}
\pgfsetdash{}{0mm}\pgfsetlinewidth{0.30mm}\pgfcircle[stroke]{\pgfxy(30.00,70.00)}{0.80mm}
\pgfputat{\pgfxy(20.00,62.00)}{\pgfbox[bottom,left]{\fontsize{9.10}{10.93}\selectfont \makebox[0pt]{$v$}}}
\pgfputat{\pgfxy(29.00,72.00)}{\pgfbox[bottom,left]{\fontsize{9.10}{10.93}\selectfont \makebox[0pt]{$u$}}}
\pgfputat{\pgfxy(25.00,80.00)}{\pgfbox[bottom,left]{\fontsize{9.10}{10.93}\selectfont \makebox[0pt]{$T$}}}
\pgfputat{\pgfxy(20.00,45.00)}{\pgfbox[bottom,left]{\fontsize{9.10}{10.93}\selectfont \makebox[0pt]{$D_v$}}}
\pgfputat{\pgfxy(37.00,56.00)}{\pgfbox[bottom,left]{\fontsize{9.10}{10.93}\selectfont \makebox[0pt]{$D_u$}}}
\pgfmoveto{\pgfxy(85.00,80.00)}\pgflineto{\pgfxy(85.00,80.00)}\pgfstroke
\pgfmoveto{\pgfxy(85.00,80.00)}\pgflineto{\pgfxy(80.00,70.00)}\pgfstroke
\pgfmoveto{\pgfxy(80.00,70.00)}\pgflineto{\pgfxy(70.00,60.00)}\pgfstroke
\pgfmoveto{\pgfxy(80.00,70.00)}\pgflineto{\pgfxy(81.00,60.00)}\pgfstroke
\pgfmoveto{\pgfxy(80.00,70.00)}\pgflineto{\pgfxy(86.00,60.00)}\pgfstroke
\pgfmoveto{\pgfxy(80.00,70.00)}\pgflineto{\pgfxy(89.00,60.00)}\pgfstroke
\pgfsetdash{{2.00mm}{1.00mm}}{0mm}\pgfsetlinewidth{0.15mm}\pgfellipse[stroke]{\pgfxy(70.00,50.00)}{\pgfxy(5.00,0.00)}{\pgfxy(0.00,10.00)}
\pgfellipse[stroke]{\pgfxy(84.49,60.80)}{\pgfxy(-3.96,-2.04)}{\pgfxy(4.56,-8.84)}
\pgfcircle[fill]{\pgfxy(70.00,60.00)}{0.80mm}
\pgfsetdash{}{0mm}\pgfsetlinewidth{0.30mm}\pgfcircle[stroke]{\pgfxy(70.00,60.00)}{0.80mm}
\pgfputat{\pgfxy(70.00,62.00)}{\pgfbox[bottom,left]{\fontsize{9.10}{10.93}\selectfont \makebox[0pt]{$v$}}}
\pgfputat{\pgfxy(79.00,72.00)}{\pgfbox[bottom,left]{\fontsize{9.10}{10.93}\selectfont \makebox[0pt]{$u$}}}
\pgfputat{\pgfxy(75.00,80.00)}{\pgfbox[bottom,left]{\fontsize{9.10}{10.93}\selectfont \makebox[0pt]{$T'$}}}
\pgfputat{\pgfxy(87.00,55.00)}{\pgfbox[bottom,left]{\fontsize{9.10}{10.93}\selectfont \makebox[0pt]{$D_v$}}}
\pgfputat{\pgfxy(70.00,46.00)}{\pgfbox[bottom,left]{\fontsize{9.10}{10.93}\selectfont \makebox[0pt]{$D_u$}}}
\pgfmoveto{\pgfxy(45.00,65.00)}\pgflineto{\pgfxy(60.00,65.00)}\pgfstroke
\pgfmoveto{\pgfxy(60.00,65.00)}\pgflineto{\pgfxy(57.20,65.70)}\pgflineto{\pgfxy(57.20,64.30)}\pgflineto{\pgfxy(60.00,65.00)}\pgfclosepath\pgffill
\pgfmoveto{\pgfxy(60.00,65.00)}\pgflineto{\pgfxy(57.20,65.70)}\pgflineto{\pgfxy(57.20,64.30)}\pgflineto{\pgfxy(60.00,65.00)}\pgfclosepath\pgfstroke
\pgfputat{\pgfxy(-3.00,69.00)}{\pgfbox[bottom,left]{\fontsize{9.10}{10.93}\selectfont \makebox[0pt][r]{level $\ell-1$}}}
\pgfputat{\pgfxy(-3.00,59.00)}{\pgfbox[bottom,left]{\fontsize{9.10}{10.93}\selectfont \makebox[0pt][r]{level $\ell$}}}
\pgfmoveto{\pgfxy(20.00,60.00)}\pgflineto{\pgfxy(21.00,50.00)}\pgfstroke
\pgfmoveto{\pgfxy(80.00,70.00)}\pgflineto{\pgfxy(84.00,60.00)}\pgfstroke
\color[rgb]{1,0,0}\pgfcircle[fill]{\pgfxy(80.00,70.00)}{0.80mm}
\color[rgb]{0,0,0}\pgfcircle[stroke]{\pgfxy(80.00,70.00)}{0.80mm}
\color[rgb]{1,0,0}\pgfcircle[fill]{\pgfxy(20.00,60.00)}{0.80mm}
\color[rgb]{0,0,0}\pgfcircle[stroke]{\pgfxy(20.00,60.00)}{0.80mm}
\pgfputat{\pgfxy(15.00,56.00)}{\pgfbox[bottom,left]{\fontsize{9.10}{10.93}\selectfont \makebox[0pt][r]{$k-1$}}}
\pgfputat{\pgfxy(85.00,69.00)}{\pgfbox[bottom,left]{\fontsize{9.10}{10.93}\selectfont $k-1$}}
\end{pgfpicture}%
\caption{A bijection from $\mathcal{A}$ to $\mathcal{D}$}
\label{fig:1to4}
\end{figure}

\subsection*{(\ref{item:e2}) $\Leftrightarrow$ (\ref{item:e4}).}
A bijection from $\mathcal{B}$ to $\mathcal{D}$ is constructed as follow: given $(T,v) \in \mathcal{B}$,
we perform the following action:
\begin{enumerate}[(a)]
\item if $v$ is the root of $T$,
the pair $(T,v)$ also belongs to $\mathcal{D}$ due to $$\outdeg(T,v) = \deg(T,v) = k.$$
\item if $v$ is not the root of $T$, find the parent vertex $u$ of $v$, the first-child $w$ of $u$ and the edge $e=(u,w)$.
As the Figure~\ref{fig:2to4}, cut and paste the edge $e$ and the subtree $D_w$ consisting of all descendants of $w$ from $u$ to $v$ such that the vertex $w$ is the first-child of $v$.
We obtain the tree $T'$ and the vertex $v$ in $T'$ satisfies
\begin{align*}
\outdeg(T',v) &= \deg(T,v) = k, & \lev(T',v) &= \lev(T,v) = \ell -1.
\end{align*}
Thus, $(T',v)$ belongs to $\mathcal{D}$.
\end{enumerate}
Since this action is reversible, it is a one-to-one correspondence.

\begin{figure}[t]
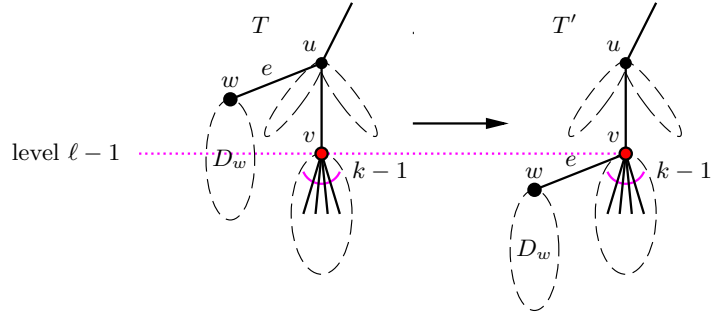

\centering
\begin{pgfpicture}{-18.98mm}{25.20mm}{70.27mm}{70.00mm}
\pgfsetxvec{\pgfpoint{0.80mm}{0mm}}
\pgfsetyvec{\pgfpoint{0mm}{0.80mm}}
\color[rgb]{0,0,0}\pgfsetlinewidth{0.30mm}\pgfsetdash{}{0mm}
\color[rgb]{1,0,1}\pgfmoveto{\pgfxy(72.00,57.00)}\pgfcurveto{\pgfxy(72.50,55.79)}{\pgfxy(73.69,55.00)}{\pgfxy(75.00,55.00)}\pgfcurveto{\pgfxy(76.31,55.00)}{\pgfxy(77.50,55.79)}{\pgfxy(78.00,57.00)}\pgfstroke
\pgfmoveto{\pgfxy(22.00,57.00)}\pgfcurveto{\pgfxy(22.50,55.79)}{\pgfxy(23.69,55.00)}{\pgfxy(25.00,55.00)}\pgfcurveto{\pgfxy(26.31,55.00)}{\pgfxy(27.50,55.79)}{\pgfxy(28.00,57.00)}\pgfstroke
\pgfsetdash{{0.30mm}{0.50mm}}{0mm}\pgfmoveto{\pgfxy(-5.00,60.00)}\pgflineto{\pgfxy(75.00,60.00)}\pgfstroke
\color[rgb]{0,0,0}\pgfsetdash{}{0mm}\pgfmoveto{\pgfxy(40.00,80.00)}\pgflineto{\pgfxy(40.00,80.00)}\pgfstroke
\pgfmoveto{\pgfxy(40.00,65.00)}\pgflineto{\pgfxy(55.00,65.00)}\pgfstroke
\pgfmoveto{\pgfxy(55.00,65.00)}\pgflineto{\pgfxy(52.20,65.70)}\pgflineto{\pgfxy(52.20,64.30)}\pgflineto{\pgfxy(55.00,65.00)}\pgfclosepath\pgffill
\pgfmoveto{\pgfxy(55.00,65.00)}\pgflineto{\pgfxy(52.20,65.70)}\pgflineto{\pgfxy(52.20,64.30)}\pgflineto{\pgfxy(55.00,65.00)}\pgfclosepath\pgfstroke
\pgfmoveto{\pgfxy(30.00,85.00)}\pgflineto{\pgfxy(25.00,75.00)}\pgfstroke
\pgfmoveto{\pgfxy(25.00,75.00)}\pgflineto{\pgfxy(25.00,60.00)}\pgfstroke
\pgfmoveto{\pgfxy(25.00,60.00)}\pgflineto{\pgfxy(22.00,50.00)}\pgfstroke
\pgfmoveto{\pgfxy(25.00,60.00)}\pgflineto{\pgfxy(24.00,50.00)}\pgfstroke
\pgfmoveto{\pgfxy(25.00,60.00)}\pgflineto{\pgfxy(28.00,50.00)}\pgfstroke
\pgfsetdash{{2.00mm}{1.00mm}}{0mm}\pgfsetlinewidth{0.15mm}\pgfellipse[stroke]{\pgfxy(25.00,50.00)}{\pgfxy(5.00,0.00)}{\pgfxy(0.00,10.00)}
\pgfputat{\pgfxy(23.00,62.00)}{\pgfbox[bottom,left]{\fontsize{9.10}{10.93}\selectfont \makebox[0pt]{$v$}}}
\pgfputat{\pgfxy(23.00,77.00)}{\pgfbox[bottom,left]{\fontsize{9.10}{10.93}\selectfont \makebox[0pt]{$u$}}}
\pgfputat{\pgfxy(15.00,80.00)}{\pgfbox[bottom,left]{\fontsize{9.10}{10.93}\selectfont \makebox[0pt]{$T$}}}
\pgfsetdash{}{0mm}\pgfsetlinewidth{0.30mm}\pgfmoveto{\pgfxy(25.00,75.00)}\pgflineto{\pgfxy(10.00,69.00)}\pgfstroke
\pgfsetdash{{2.00mm}{1.00mm}}{0mm}\pgfsetlinewidth{0.15mm}\pgfellipse[stroke]{\pgfxy(10.00,59.00)}{\pgfxy(4.00,0.00)}{\pgfxy(0.00,10.00)}
\pgfellipse[stroke]{\pgfxy(20.39,69.07)}{\pgfxy(1.20,-0.90)}{\pgfxy(4.60,6.17)}
\pgfellipse[stroke]{\pgfxy(29.39,69.07)}{\pgfxy(-1.20,-0.90)}{\pgfxy(4.60,-6.17)}
\pgfcircle[fill]{\pgfxy(25.00,75.00)}{0.80mm}
\pgfputat{\pgfxy(16.00,73.00)}{\pgfbox[bottom,left]{\fontsize{9.10}{10.93}\selectfont \makebox[0pt]{$e$}}}
\pgfputat{\pgfxy(10.00,58.00)}{\pgfbox[bottom,left]{\fontsize{9.10}{10.93}\selectfont \makebox[0pt]{$D_w$}}}
\pgfsetdash{}{0mm}\pgfsetlinewidth{0.30mm}\pgfmoveto{\pgfxy(80.00,85.00)}\pgflineto{\pgfxy(75.00,75.00)}\pgfstroke
\pgfmoveto{\pgfxy(75.00,75.00)}\pgflineto{\pgfxy(75.00,60.00)}\pgfstroke
\pgfmoveto{\pgfxy(75.00,60.00)}\pgflineto{\pgfxy(72.00,50.00)}\pgfstroke
\pgfmoveto{\pgfxy(75.00,60.00)}\pgflineto{\pgfxy(74.00,50.00)}\pgfstroke
\pgfmoveto{\pgfxy(75.00,60.00)}\pgflineto{\pgfxy(78.00,50.00)}\pgfstroke
\pgfsetdash{{2.00mm}{1.00mm}}{0mm}\pgfsetlinewidth{0.15mm}\pgfellipse[stroke]{\pgfxy(75.00,50.00)}{\pgfxy(5.00,0.00)}{\pgfxy(0.00,10.00)}
\pgfputat{\pgfxy(73.00,62.00)}{\pgfbox[bottom,left]{\fontsize{9.10}{10.93}\selectfont \makebox[0pt]{$v$}}}
\pgfputat{\pgfxy(73.00,77.00)}{\pgfbox[bottom,left]{\fontsize{9.10}{10.93}\selectfont \makebox[0pt]{$u$}}}
\pgfputat{\pgfxy(65.00,80.00)}{\pgfbox[bottom,left]{\fontsize{9.10}{10.93}\selectfont \makebox[0pt]{$T'$}}}
\pgfellipse[stroke]{\pgfxy(70.39,69.07)}{\pgfxy(1.20,-0.90)}{\pgfxy(4.60,6.17)}
\pgfellipse[stroke]{\pgfxy(79.39,69.07)}{\pgfxy(-1.20,-0.90)}{\pgfxy(4.60,-6.17)}
\pgfcircle[fill]{\pgfxy(75.00,75.00)}{0.80mm}
\pgfsetdash{}{0mm}\pgfsetlinewidth{0.30mm}\pgfmoveto{\pgfxy(75.00,60.00)}\pgflineto{\pgfxy(60.00,54.00)}\pgfstroke
\pgfputat{\pgfxy(66.00,58.00)}{\pgfbox[bottom,left]{\fontsize{9.10}{10.93}\selectfont \makebox[0pt]{$e$}}}
\pgfsetdash{{2.00mm}{1.00mm}}{0mm}\pgfsetlinewidth{0.15mm}\pgfellipse[stroke]{\pgfxy(60.00,44.00)}{\pgfxy(4.00,0.00)}{\pgfxy(0.00,10.00)}
\pgfputat{\pgfxy(60.00,43.00)}{\pgfbox[bottom,left]{\fontsize{9.10}{10.93}\selectfont \makebox[0pt]{$D_w$}}}
\pgfputat{\pgfxy(-8.00,59.00)}{\pgfbox[bottom,left]{\fontsize{9.10}{10.93}\selectfont \makebox[0pt][r]{level $\ell-1$}}}
\pgfputat{\pgfxy(30.00,56.00)}{\pgfbox[bottom,left]{\fontsize{9.10}{10.93}\selectfont $k-1$}}
\pgfputat{\pgfxy(80.00,56.00)}{\pgfbox[bottom,left]{\fontsize{9.10}{10.93}\selectfont $k-1$}}
\pgfsetdash{}{0mm}\pgfsetlinewidth{0.30mm}\pgfmoveto{\pgfxy(25.00,60.00)}\pgflineto{\pgfxy(26.00,50.00)}\pgfstroke
\pgfmoveto{\pgfxy(75.00,60.00)}\pgflineto{\pgfxy(76.00,50.00)}\pgfstroke
\color[rgb]{1,0,0}\pgfcircle[fill]{\pgfxy(75.00,60.00)}{0.80mm}
\color[rgb]{0,0,0}\pgfcircle[stroke]{\pgfxy(75.00,60.00)}{0.80mm}
\color[rgb]{1,0,0}\pgfcircle[fill]{\pgfxy(25.00,60.00)}{0.80mm}
\color[rgb]{0,0,0}\pgfcircle[stroke]{\pgfxy(25.00,60.00)}{0.80mm}
\pgfcircle[fill]{\pgfxy(10.00,69.00)}{0.80mm}
\pgfcircle[stroke]{\pgfxy(10.00,69.00)}{0.80mm}
\pgfcircle[fill]{\pgfxy(60.00,54.00)}{0.80mm}
\pgfcircle[stroke]{\pgfxy(60.00,54.00)}{0.80mm}
\pgfputat{\pgfxy(60.00,56.00)}{\pgfbox[bottom,left]{\fontsize{9.10}{10.93}\selectfont \makebox[0pt]{$w$}}}
\pgfputat{\pgfxy(10.00,71.00)}{\pgfbox[bottom,left]{\fontsize{9.10}{10.93}\selectfont \makebox[0pt]{$w$}}}
\end{pgfpicture}%
\caption{A bijection from $\mathcal{B}$ to $\mathcal{D}$ if $v$ is not the root of $T$}
\label{fig:2to4}
\end{figure}

\subsection*{(\ref{item:e3}) $\Leftrightarrow$ (\ref{item:e4}).}
A bijection from $\mathcal{C}$ to $\mathcal{D}$ is constructed as follows:
given $(T,v) \in \mathcal{C}$, find the parent vertex $u$ of $v$.
As the Figure~\ref{fig:3to4}, consider the subtree $D_{uv}$ of $u$ on the right of the edge $(u,v)$.
By cutting and pasting the subtrees $D_{uv}$ from $u$ to $v$ we get the new tree $T'$ with the vertex $u$ such that
\begin{align*}
\outdeg(T',u) &= \eld(T,v) + 1 = k, & \lev(T',u) &= \lev(T,v) -1 = \ell -1,
\end{align*}
where $\eld(T,v)$ means the number of elder siblings of a vertex $v$ in a tree $T$.
Thus, $(T',u)$ belongs to $\mathcal{D}$. Since this interchanging action is reversible, it is a one-to-one correspondence.

\begin{figure}[t]
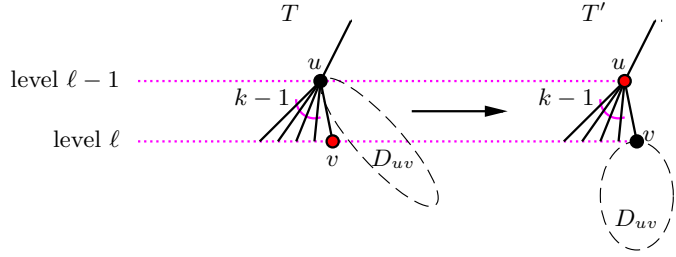

\centering
\begin{pgfpicture}{-14.98mm}{31.60mm}{72.40mm}{68.51mm}
\pgfsetxvec{\pgfpoint{0.80mm}{0mm}}
\pgfsetyvec{\pgfpoint{0mm}{0.80mm}}
\color[rgb]{0,0,0}\pgfsetlinewidth{0.30mm}\pgfsetdash{}{0mm}
\color[rgb]{1,0,1}\pgfmoveto{\pgfxy(76.00,67.00)}\pgfcurveto{\pgfxy(75.94,65.67)}{\pgfxy(76.75,64.46)}{\pgfxy(78.00,64.00)}\pgfcurveto{\pgfxy(78.65,63.76)}{\pgfxy(79.35,63.76)}{\pgfxy(80.00,64.00)}\pgfstroke
\pgfmoveto{\pgfxy(26.00,67.00)}\pgfcurveto{\pgfxy(25.94,65.67)}{\pgfxy(26.75,64.46)}{\pgfxy(28.00,64.00)}\pgfcurveto{\pgfxy(28.65,63.76)}{\pgfxy(29.35,63.76)}{\pgfxy(30.00,64.00)}\pgfstroke
\pgfsetdash{{0.30mm}{0.50mm}}{0mm}\pgfmoveto{\pgfxy(0.00,60.00)}\pgflineto{\pgfxy(82.00,60.00)}\pgfstroke
\pgfmoveto{\pgfxy(0.00,70.00)}\pgflineto{\pgfxy(80.00,70.00)}\pgfstroke
\color[rgb]{0,0,0}\pgfsetdash{}{0mm}\pgfmoveto{\pgfxy(35.00,80.00)}\pgflineto{\pgfxy(35.00,80.00)}\pgfstroke
\pgfmoveto{\pgfxy(35.00,80.00)}\pgflineto{\pgfxy(30.00,70.00)}\pgfstroke
\pgfmoveto{\pgfxy(30.00,70.00)}\pgflineto{\pgfxy(20.00,60.00)}\pgfstroke
\pgfmoveto{\pgfxy(30.00,70.00)}\pgflineto{\pgfxy(32.00,60.00)}\pgfstroke
\pgfmoveto{\pgfxy(30.00,70.00)}\pgflineto{\pgfxy(26.00,60.00)}\pgfstroke
\pgfmoveto{\pgfxy(30.00,70.00)}\pgflineto{\pgfxy(29.00,60.00)}\pgfstroke
\pgfmoveto{\pgfxy(30.00,70.00)}\pgflineto{\pgfxy(23.00,60.00)}\pgfstroke
\pgfsetdash{{2.00mm}{1.00mm}}{0mm}\pgfsetlinewidth{0.15mm}\pgfellipse[stroke]{\pgfxy(39.67,59.86)}{\pgfxy(-2.74,-2.46)}{\pgfxy(9.28,-10.35)}
\pgfcircle[fill]{\pgfxy(30.00,70.00)}{0.80mm}
\pgfsetdash{}{0mm}\pgfsetlinewidth{0.30mm}\pgfcircle[stroke]{\pgfxy(30.00,70.00)}{0.80mm}
\pgfputat{\pgfxy(32.00,56.00)}{\pgfbox[bottom,left]{\fontsize{9.10}{10.93}\selectfont \makebox[0pt]{$v$}}}
\pgfputat{\pgfxy(29.00,72.00)}{\pgfbox[bottom,left]{\fontsize{9.10}{10.93}\selectfont \makebox[0pt]{$u$}}}
\pgfputat{\pgfxy(25.00,80.00)}{\pgfbox[bottom,left]{\fontsize{9.10}{10.93}\selectfont \makebox[0pt]{$T$}}}
\pgfputat{\pgfxy(42.00,56.00)}{\pgfbox[bottom,left]{\fontsize{9.10}{10.93}\selectfont \makebox[0pt]{$D_{uv}$}}}
\pgfmoveto{\pgfxy(86.00,80.00)}\pgflineto{\pgfxy(86.00,80.00)}\pgfstroke
\pgfputat{\pgfxy(75.00,80.00)}{\pgfbox[bottom,left]{\fontsize{9.10}{10.93}\selectfont \makebox[0pt]{$T'$}}}
\pgfmoveto{\pgfxy(45.00,65.00)}\pgflineto{\pgfxy(60.00,65.00)}\pgfstroke
\pgfmoveto{\pgfxy(60.00,65.00)}\pgflineto{\pgfxy(57.20,65.70)}\pgflineto{\pgfxy(57.20,64.30)}\pgflineto{\pgfxy(60.00,65.00)}\pgfclosepath\pgffill
\pgfmoveto{\pgfxy(60.00,65.00)}\pgflineto{\pgfxy(57.20,65.70)}\pgflineto{\pgfxy(57.20,64.30)}\pgflineto{\pgfxy(60.00,65.00)}\pgfclosepath\pgfstroke
\pgfputat{\pgfxy(-3.00,69.00)}{\pgfbox[bottom,left]{\fontsize{9.10}{10.93}\selectfont \makebox[0pt][r]{level $\ell-1$}}}
\pgfputat{\pgfxy(-3.00,59.00)}{\pgfbox[bottom,left]{\fontsize{9.10}{10.93}\selectfont \makebox[0pt][r]{level $\ell$}}}
\color[rgb]{1,0,0}\pgfcircle[fill]{\pgfxy(32.00,60.00)}{0.80mm}
\color[rgb]{0,0,0}\pgfcircle[stroke]{\pgfxy(32.00,60.00)}{0.80mm}
\pgfmoveto{\pgfxy(85.00,80.00)}\pgflineto{\pgfxy(85.00,80.00)}\pgfstroke
\pgfmoveto{\pgfxy(85.00,80.00)}\pgflineto{\pgfxy(80.00,70.00)}\pgfstroke
\pgfmoveto{\pgfxy(80.00,70.00)}\pgflineto{\pgfxy(70.00,60.00)}\pgfstroke
\pgfmoveto{\pgfxy(80.00,70.00)}\pgflineto{\pgfxy(82.00,60.00)}\pgfstroke
\pgfmoveto{\pgfxy(80.00,70.00)}\pgflineto{\pgfxy(76.00,60.00)}\pgfstroke
\pgfmoveto{\pgfxy(80.00,70.00)}\pgflineto{\pgfxy(79.00,60.00)}\pgfstroke
\pgfmoveto{\pgfxy(80.00,70.00)}\pgflineto{\pgfxy(73.00,60.00)}\pgfstroke
\pgfputat{\pgfxy(84.00,60.00)}{\pgfbox[bottom,left]{\fontsize{9.10}{10.93}\selectfont \makebox[0pt]{$v$}}}
\pgfputat{\pgfxy(79.00,72.00)}{\pgfbox[bottom,left]{\fontsize{9.10}{10.93}\selectfont \makebox[0pt]{$u$}}}
\pgfsetdash{{2.00mm}{1.00mm}}{0mm}\pgfsetlinewidth{0.15mm}\pgfellipse[stroke]{\pgfxy(82.00,51.00)}{\pgfxy(6.00,0.00)}{\pgfxy(0.00,9.00)}
\color[rgb]{1,0,0}\pgfcircle[fill]{\pgfxy(80.00,70.00)}{0.80mm}
\color[rgb]{0,0,0}\pgfsetdash{}{0mm}\pgfsetlinewidth{0.30mm}\pgfcircle[stroke]{\pgfxy(80.00,70.00)}{0.80mm}
\pgfcircle[fill]{\pgfxy(82.00,60.00)}{0.80mm}
\pgfcircle[stroke]{\pgfxy(82.00,60.00)}{0.80mm}
\pgfputat{\pgfxy(25.00,66.00)}{\pgfbox[bottom,left]{\fontsize{9.10}{10.93}\selectfont \makebox[0pt][r]{$k-1$}}}
\pgfputat{\pgfxy(75.00,66.00)}{\pgfbox[bottom,left]{\fontsize{9.10}{10.93}\selectfont \makebox[0pt][r]{$k-1$}}}
\pgfputat{\pgfxy(82.00,46.00)}{\pgfbox[bottom,left]{\fontsize{9.10}{10.93}\selectfont \makebox[0pt]{$D_{uv}$}}}
\end{pgfpicture}%
\caption{A bijection from $\mathcal{C}$ to $\mathcal{D}$}
\label{fig:3to4}
\end{figure}

\section{Proof of Theorem~\ref{thm:main}: Enumeration}
\label{sec:enumeration}

By putting $\ell+1$ in place of $\ell$ in \eqref{item:c4} of Theorem~\ref{thm:main},
it suffices to show that
for any two nonnegative integers $k$ and $\ell$,
the number of vertices of outdegree $k$ at level $\ell$ among trees in $\T_n$ is
$$\frac{k+2\ell}{2n-k} \binom{2n -k }{ n +\ell}.$$
The following lemma gives a cumulative counting in $k$ and $\ell$.

\begin{lem}[Main lemma]
Given $n \ge 1$,
for any two nonnegative integers $k$ and $\ell$,
the number of vertices of outdegree at least $k$ and at level at least $\ell$ among trees in $\T_n$ is
\begin{align}
\label{eq:lem}
\binom{2n -k }{ n+\ell}.
\end{align}

\end{lem}

\begin{proof}
Let $\mathcal{V}$ be the set of $(T,v) \in \TV_n$ such that $v$ is a vertex of outdegree at least $k$ ant at level at least $\ell$ in $T$. Let $\mathcal{L}$ be the set of \emph{lattice paths} of length $(2n-k)$ from $(k, k)$ to $(2n, -2\ell)$, consisting of $(n-k-\ell)$ up-steps along the vector $(1,1)$ and $(n+\ell)$ down-steps along the vector $(1,-1)$.
Since
$$\# \mathcal{L} = \binom{2n - k}{n -k -\ell,~ n+\ell} = \binom{2n -k }{ n+\ell },$$
it is enough to construct a bijection $\Phi$ between $\mathcal{V}$ and $\mathcal{L}$.

Before constructing the bijection we first introduce two well-known bijections $\varphi$ and $\psi$ between rooted ordered trees and Dyck paths.

The bijection $\varphi$ corresponds a tree to a Dyck path by recording the steps when the tree is traversed in preorder. Here we record an up-step when we go down an edge and a down-step when going up.
An example of the bijection $\varphi$ is shown in the Figure~\ref{fig:varphi}.
\begin{figure}[t]
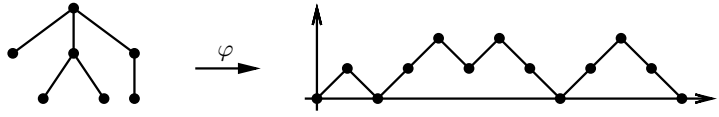

\centering
\begin{pgfpicture}{9.44mm}{48.40mm}{106.00mm}{66.56mm}
\pgfsetxvec{\pgfpoint{0.80mm}{0mm}}
\pgfsetyvec{\pgfpoint{0mm}{0.80mm}}
\color[rgb]{0,0,0}\pgfsetlinewidth{0.30mm}\pgfsetdash{}{0mm}
\pgfmoveto{\pgfxy(45.00,70.00)}\pgflineto{\pgfxy(55.00,70.00)}\pgfstroke
\pgfmoveto{\pgfxy(55.00,70.00)}\pgflineto{\pgfxy(52.20,70.70)}\pgflineto{\pgfxy(52.20,69.30)}\pgflineto{\pgfxy(55.00,70.00)}\pgfclosepath\pgffill
\pgfmoveto{\pgfxy(55.00,70.00)}\pgflineto{\pgfxy(52.20,70.70)}\pgflineto{\pgfxy(52.20,69.30)}\pgflineto{\pgfxy(55.00,70.00)}\pgfclosepath\pgfstroke
\pgfmoveto{\pgfxy(65.00,63.00)}\pgflineto{\pgfxy(65.00,80.00)}\pgfstroke
\pgfmoveto{\pgfxy(65.00,80.00)}\pgflineto{\pgfxy(64.30,77.20)}\pgflineto{\pgfxy(65.00,80.00)}\pgflineto{\pgfxy(65.70,77.20)}\pgflineto{\pgfxy(65.00,80.00)}\pgfclosepath\pgffill
\pgfmoveto{\pgfxy(65.00,80.00)}\pgflineto{\pgfxy(64.30,77.20)}\pgflineto{\pgfxy(65.00,80.00)}\pgflineto{\pgfxy(65.70,77.20)}\pgflineto{\pgfxy(65.00,80.00)}\pgfclosepath\pgfstroke
\pgfmoveto{\pgfxy(63.00,65.00)}\pgflineto{\pgfxy(130.00,65.00)}\pgfstroke
\pgfmoveto{\pgfxy(130.00,65.00)}\pgflineto{\pgfxy(127.20,65.70)}\pgflineto{\pgfxy(129.30,65.00)}\pgflineto{\pgfxy(127.20,64.30)}\pgflineto{\pgfxy(130.00,65.00)}\pgfclosepath\pgffill
\pgfmoveto{\pgfxy(130.00,65.00)}\pgflineto{\pgfxy(127.20,65.70)}\pgflineto{\pgfxy(129.30,65.00)}\pgflineto{\pgfxy(127.20,64.30)}\pgflineto{\pgfxy(130.00,65.00)}\pgfclosepath\pgfstroke
\pgfmoveto{\pgfxy(65.00,65.00)}\pgflineto{\pgfxy(70.00,70.00)}\pgflineto{\pgfxy(75.00,65.00)}\pgflineto{\pgfxy(80.00,70.00)}\pgflineto{\pgfxy(85.00,75.00)}\pgflineto{\pgfxy(90.00,70.00)}\pgflineto{\pgfxy(95.00,75.00)}\pgflineto{\pgfxy(100.00,70.00)}\pgflineto{\pgfxy(105.00,65.00)}\pgflineto{\pgfxy(110.00,70.00)}\pgflineto{\pgfxy(115.00,75.00)}\pgflineto{\pgfxy(120.00,70.00)}\pgflineto{\pgfxy(125.00,65.00)}\pgfstroke
\pgfputat{\pgfxy(50.00,72.00)}{\pgfbox[bottom,left]{\fontsize{9.10}{10.93}\selectfont \makebox[0pt]{$\varphi$}}}
\pgfmoveto{\pgfxy(25.00,80.00)}\pgflineto{\pgfxy(15.00,72.50)}\pgfstroke
\pgfmoveto{\pgfxy(25.00,80.00)}\pgflineto{\pgfxy(25.00,72.50)}\pgfstroke
\pgfmoveto{\pgfxy(25.00,80.00)}\pgflineto{\pgfxy(35.00,72.50)}\pgfstroke
\pgfmoveto{\pgfxy(25.00,72.50)}\pgflineto{\pgfxy(20.00,65.00)}\pgfstroke
\pgfmoveto{\pgfxy(25.00,72.50)}\pgflineto{\pgfxy(30.00,65.00)}\pgfstroke
\pgfmoveto{\pgfxy(35.00,72.50)}\pgflineto{\pgfxy(35.00,65.00)}\pgfstroke
\pgfcircle[fill]{\pgfxy(30.00,65.00)}{0.56mm}
\pgfcircle[stroke]{\pgfxy(30.00,65.00)}{0.56mm}
\pgfcircle[fill]{\pgfxy(35.00,65.00)}{0.56mm}
\pgfcircle[stroke]{\pgfxy(35.00,65.00)}{0.56mm}
\pgfcircle[fill]{\pgfxy(20.00,65.00)}{0.56mm}
\pgfcircle[stroke]{\pgfxy(20.00,65.00)}{0.56mm}
\pgfcircle[fill]{\pgfxy(25.00,80.00)}{0.56mm}
\pgfcircle[stroke]{\pgfxy(25.00,80.00)}{0.56mm}
\pgfcircle[fill]{\pgfxy(25.00,72.50)}{0.56mm}
\pgfcircle[stroke]{\pgfxy(25.00,72.50)}{0.56mm}
\pgfcircle[fill]{\pgfxy(35.00,72.50)}{0.56mm}
\pgfcircle[stroke]{\pgfxy(35.00,72.50)}{0.56mm}
\pgfcircle[fill]{\pgfxy(15.00,72.50)}{0.56mm}
\pgfcircle[stroke]{\pgfxy(15.00,72.50)}{0.56mm}
\pgfcircle[fill]{\pgfxy(125.00,65.00)}{0.56mm}
\pgfcircle[stroke]{\pgfxy(125.00,65.00)}{0.56mm}
\pgfcircle[fill]{\pgfxy(120.00,70.00)}{0.56mm}
\pgfcircle[stroke]{\pgfxy(120.00,70.00)}{0.56mm}
\pgfcircle[fill]{\pgfxy(115.00,75.00)}{0.56mm}
\pgfcircle[stroke]{\pgfxy(115.00,75.00)}{0.56mm}
\pgfcircle[fill]{\pgfxy(110.00,70.00)}{0.56mm}
\pgfcircle[stroke]{\pgfxy(110.00,70.00)}{0.56mm}
\pgfcircle[fill]{\pgfxy(105.00,65.00)}{0.56mm}
\pgfcircle[stroke]{\pgfxy(105.00,65.00)}{0.56mm}
\pgfcircle[fill]{\pgfxy(100.00,70.00)}{0.56mm}
\pgfcircle[stroke]{\pgfxy(100.00,70.00)}{0.56mm}
\pgfcircle[fill]{\pgfxy(95.00,75.00)}{0.56mm}
\pgfcircle[stroke]{\pgfxy(95.00,75.00)}{0.56mm}
\pgfcircle[fill]{\pgfxy(90.00,70.00)}{0.56mm}
\pgfcircle[stroke]{\pgfxy(90.00,70.00)}{0.56mm}
\pgfcircle[fill]{\pgfxy(85.00,75.00)}{0.56mm}
\pgfcircle[stroke]{\pgfxy(85.00,75.00)}{0.56mm}
\pgfcircle[fill]{\pgfxy(80.00,70.00)}{0.56mm}
\pgfcircle[stroke]{\pgfxy(80.00,70.00)}{0.56mm}
\pgfcircle[fill]{\pgfxy(75.00,65.00)}{0.56mm}
\pgfcircle[stroke]{\pgfxy(75.00,65.00)}{0.56mm}
\pgfcircle[fill]{\pgfxy(70.00,70.00)}{0.56mm}
\pgfcircle[stroke]{\pgfxy(70.00,70.00)}{0.56mm}
\pgfcircle[fill]{\pgfxy(65.00,65.00)}{0.56mm}
\pgfcircle[stroke]{\pgfxy(65.00,65.00)}{0.56mm}
\end{pgfpicture}%
\caption{The bijection $\varphi$}
\label{fig:varphi}
\end{figure}

The bijection $\psi$ corresponds a tree to a Dyck path by recording the steps when the tree is traversed in preorder. Here, whenever we meet a vertex of outdegree $k$, except the last leaf, we record $k$ up-steps and one down-step.
An example of the bijection $\psi$ is shown in the Figure~\ref{fig:psi}.
\begin{figure}[t]
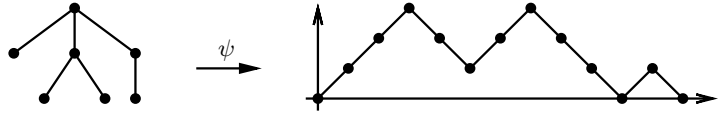

\centering
\begin{pgfpicture}{9.44mm}{48.40mm}{106.00mm}{66.56mm}
\pgfsetxvec{\pgfpoint{0.80mm}{0mm}}
\pgfsetyvec{\pgfpoint{0mm}{0.80mm}}
\color[rgb]{0,0,0}\pgfsetlinewidth{0.30mm}\pgfsetdash{}{0mm}
\pgfmoveto{\pgfxy(45.00,70.00)}\pgflineto{\pgfxy(55.00,70.00)}\pgfstroke
\pgfmoveto{\pgfxy(55.00,70.00)}\pgflineto{\pgfxy(52.20,70.70)}\pgflineto{\pgfxy(52.20,69.30)}\pgflineto{\pgfxy(55.00,70.00)}\pgfclosepath\pgffill
\pgfmoveto{\pgfxy(55.00,70.00)}\pgflineto{\pgfxy(52.20,70.70)}\pgflineto{\pgfxy(52.20,69.30)}\pgflineto{\pgfxy(55.00,70.00)}\pgfclosepath\pgfstroke
\pgfmoveto{\pgfxy(25.00,80.00)}\pgflineto{\pgfxy(15.00,72.50)}\pgfstroke
\pgfmoveto{\pgfxy(25.00,80.00)}\pgflineto{\pgfxy(25.00,72.50)}\pgfstroke
\pgfmoveto{\pgfxy(25.00,80.00)}\pgflineto{\pgfxy(35.00,72.50)}\pgfstroke
\pgfmoveto{\pgfxy(25.00,72.50)}\pgflineto{\pgfxy(20.00,65.00)}\pgfstroke
\pgfmoveto{\pgfxy(25.00,72.50)}\pgflineto{\pgfxy(30.00,65.00)}\pgfstroke
\pgfmoveto{\pgfxy(65.00,63.00)}\pgflineto{\pgfxy(65.00,80.00)}\pgfstroke
\pgfmoveto{\pgfxy(65.00,80.00)}\pgflineto{\pgfxy(64.30,77.20)}\pgflineto{\pgfxy(65.00,80.00)}\pgflineto{\pgfxy(65.70,77.20)}\pgflineto{\pgfxy(65.00,80.00)}\pgfclosepath\pgffill
\pgfmoveto{\pgfxy(65.00,80.00)}\pgflineto{\pgfxy(64.30,77.20)}\pgflineto{\pgfxy(65.00,80.00)}\pgflineto{\pgfxy(65.70,77.20)}\pgflineto{\pgfxy(65.00,80.00)}\pgfclosepath\pgfstroke
\pgfmoveto{\pgfxy(63.00,65.00)}\pgflineto{\pgfxy(130.00,65.00)}\pgfstroke
\pgfmoveto{\pgfxy(130.00,65.00)}\pgflineto{\pgfxy(127.20,65.70)}\pgflineto{\pgfxy(129.30,65.00)}\pgflineto{\pgfxy(127.20,64.30)}\pgflineto{\pgfxy(130.00,65.00)}\pgfclosepath\pgffill
\pgfmoveto{\pgfxy(130.00,65.00)}\pgflineto{\pgfxy(127.20,65.70)}\pgflineto{\pgfxy(129.30,65.00)}\pgflineto{\pgfxy(127.20,64.30)}\pgflineto{\pgfxy(130.00,65.00)}\pgfclosepath\pgfstroke
\pgfputat{\pgfxy(50.00,72.00)}{\pgfbox[bottom,left]{\fontsize{9.10}{10.93}\selectfont \makebox[0pt]{$\psi$}}}
\pgfmoveto{\pgfxy(35.00,72.50)}\pgflineto{\pgfxy(35.00,65.00)}\pgfstroke
\pgfmoveto{\pgfxy(65.00,65.00)}\pgflineto{\pgfxy(70.00,70.00)}\pgflineto{\pgfxy(75.00,75.00)}\pgflineto{\pgfxy(80.00,80.00)}\pgflineto{\pgfxy(85.00,75.00)}\pgflineto{\pgfxy(90.00,70.00)}\pgflineto{\pgfxy(95.00,75.00)}\pgflineto{\pgfxy(100.00,80.00)}\pgflineto{\pgfxy(105.00,75.00)}\pgflineto{\pgfxy(110.00,70.00)}\pgflineto{\pgfxy(115.00,65.00)}\pgflineto{\pgfxy(120.00,70.00)}\pgflineto{\pgfxy(125.00,65.00)}\pgfstroke
\pgfcircle[fill]{\pgfxy(65.00,65.00)}{0.56mm}
\pgfcircle[stroke]{\pgfxy(65.00,65.00)}{0.56mm}
\pgfcircle[fill]{\pgfxy(85.00,75.00)}{0.56mm}
\pgfcircle[stroke]{\pgfxy(85.00,75.00)}{0.56mm}
\pgfcircle[fill]{\pgfxy(80.00,80.00)}{0.56mm}
\pgfcircle[stroke]{\pgfxy(80.00,80.00)}{0.56mm}
\pgfcircle[fill]{\pgfxy(75.00,75.00)}{0.56mm}
\pgfcircle[stroke]{\pgfxy(75.00,75.00)}{0.56mm}
\pgfcircle[fill]{\pgfxy(90.00,70.00)}{0.56mm}
\pgfcircle[stroke]{\pgfxy(90.00,70.00)}{0.56mm}
\pgfcircle[fill]{\pgfxy(70.00,70.00)}{0.56mm}
\pgfcircle[stroke]{\pgfxy(70.00,70.00)}{0.56mm}
\pgfcircle[fill]{\pgfxy(30.00,65.00)}{0.56mm}
\pgfcircle[stroke]{\pgfxy(30.00,65.00)}{0.56mm}
\pgfcircle[fill]{\pgfxy(35.00,65.00)}{0.56mm}
\pgfcircle[stroke]{\pgfxy(35.00,65.00)}{0.56mm}
\pgfcircle[fill]{\pgfxy(20.00,65.00)}{0.56mm}
\pgfcircle[stroke]{\pgfxy(20.00,65.00)}{0.56mm}
\pgfcircle[fill]{\pgfxy(25.00,80.00)}{0.56mm}
\pgfcircle[stroke]{\pgfxy(25.00,80.00)}{0.56mm}
\pgfcircle[fill]{\pgfxy(25.00,72.50)}{0.56mm}
\pgfcircle[stroke]{\pgfxy(25.00,72.50)}{0.56mm}
\pgfcircle[fill]{\pgfxy(35.00,72.50)}{0.56mm}
\pgfcircle[stroke]{\pgfxy(35.00,72.50)}{0.56mm}
\pgfcircle[fill]{\pgfxy(15.00,72.50)}{0.56mm}
\pgfcircle[stroke]{\pgfxy(15.00,72.50)}{0.56mm}
\pgfcircle[fill]{\pgfxy(125.00,65.00)}{0.56mm}
\pgfcircle[stroke]{\pgfxy(125.00,65.00)}{0.56mm}
\pgfcircle[fill]{\pgfxy(120.00,70.00)}{0.56mm}
\pgfcircle[stroke]{\pgfxy(120.00,70.00)}{0.56mm}
\pgfcircle[fill]{\pgfxy(115.00,65.00)}{0.56mm}
\pgfcircle[stroke]{\pgfxy(115.00,65.00)}{0.56mm}
\pgfcircle[fill]{\pgfxy(110.00,70.00)}{0.56mm}
\pgfcircle[stroke]{\pgfxy(110.00,70.00)}{0.56mm}
\pgfcircle[fill]{\pgfxy(105.00,75.00)}{0.56mm}
\pgfcircle[stroke]{\pgfxy(105.00,75.00)}{0.56mm}
\pgfcircle[fill]{\pgfxy(100.00,80.00)}{0.56mm}
\pgfcircle[stroke]{\pgfxy(100.00,80.00)}{0.56mm}
\pgfcircle[fill]{\pgfxy(95.00,75.00)}{0.56mm}
\pgfcircle[stroke]{\pgfxy(95.00,75.00)}{0.56mm}
\end{pgfpicture}%
\caption{The bijection $\psi$}
\label{fig:psi}
\end{figure}

Using two bijections $\varphi$ and $\psi$,
we will construct another bijection $\Phi$ between $\mathcal{V}$ to $\mathcal{L}$:
given $(T,v) \in \mathcal{V}$, 
let $k' (\ge k)$ be the number of children of $v$ in $T$ and 
let $\ell' (\ge \ell)$ be the level of $v$ in $T$.
We decompose the tree $T$ into $\ell'+2$ subtrees:
the subtree $D_v$ consisting of descendants of $v$,
$\ell'$ subtrees $R_1$, $R_2, \dots, R_{\ell'}$ on the right hand side of the path from $v$ to the root $r$ of $T$, and the remaining tree $L$ as illustrated in the Figure~\ref{fig:decomposition}.
\begin{figure}[t]
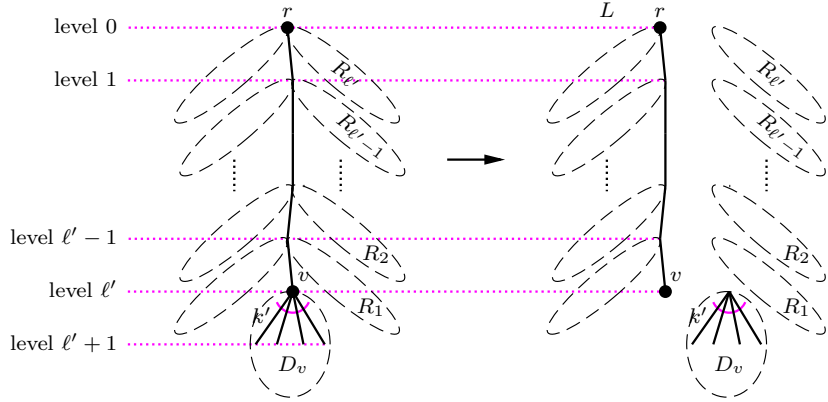

\centering
\begin{pgfpicture}{-21.70mm}{-2.00mm}{86.72mm}{54.60mm}
\pgfsetxvec{\pgfpoint{0.70mm}{0mm}}
\pgfsetyvec{\pgfpoint{0mm}{0.70mm}}
\color[rgb]{0,0,0}\pgfsetlinewidth{0.30mm}\pgfsetdash{}{0mm}
\color[rgb]{1,0,1}\pgfsetdash{{0.30mm}{0.50mm}}{0mm}\pgfmoveto{\pgfxy(90.00,70.00)}\pgflineto{\pgfxy(-10.00,70.00)}\pgfstroke
\pgfmoveto{\pgfxy(91.00,60.00)}\pgflineto{\pgfxy(-10.00,60.00)}\pgfstroke
\pgfmoveto{\pgfxy(90.00,30.00)}\pgflineto{\pgfxy(-10.00,30.00)}\pgfstroke
\pgfmoveto{\pgfxy(91.00,20.00)}\pgflineto{\pgfxy(-10.00,20.00)}\pgfstroke
\pgfmoveto{\pgfxy(27.00,10.00)}\pgflineto{\pgfxy(-10.00,10.00)}\pgfstroke
\color[rgb]{0,0,0}\pgfsetdash{}{0mm}\pgfmoveto{\pgfxy(50.00,45.00)}\pgflineto{\pgfxy(60.00,45.00)}\pgfstroke
\pgfmoveto{\pgfxy(60.00,45.00)}\pgflineto{\pgfxy(57.20,45.70)}\pgflineto{\pgfxy(57.20,44.30)}\pgflineto{\pgfxy(60.00,45.00)}\pgfclosepath\pgffill
\pgfmoveto{\pgfxy(60.00,45.00)}\pgflineto{\pgfxy(57.20,45.70)}\pgflineto{\pgfxy(57.20,44.30)}\pgflineto{\pgfxy(60.00,45.00)}\pgfclosepath\pgfstroke
\pgfmoveto{\pgfxy(23.00,40.00)}\pgflineto{\pgfxy(23.00,40.00)}\pgfstroke
\pgfmoveto{\pgfxy(21.00,40.00)}\pgflineto{\pgfxy(21.00,50.00)}\pgfstroke
\pgfmoveto{\pgfxy(21.00,50.00)}\pgflineto{\pgfxy(21.00,60.00)}\pgfstroke
\pgfmoveto{\pgfxy(21.00,60.00)}\pgflineto{\pgfxy(20.00,70.00)}\pgfstroke
\pgfmoveto{\pgfxy(20.00,30.00)}\pgflineto{\pgfxy(21.00,40.00)}\pgfstroke
\pgfmoveto{\pgfxy(21.00,20.00)}\pgflineto{\pgfxy(20.00,30.00)}\pgfstroke
\pgfsetdash{{2.00mm}{1.00mm}}{0mm}\pgfsetlinewidth{0.15mm}\pgfellipse[stroke]{\pgfxy(30.41,21.08)}{\pgfxy(-1.86,-2.19)}{\pgfxy(10.46,-8.89)}
\pgfellipse[stroke]{\pgfxy(9.41,21.08)}{\pgfxy(1.86,-2.19)}{\pgfxy(10.46,8.89)}
\pgfellipse[stroke]{\pgfxy(31.41,31.08)}{\pgfxy(-1.86,-2.19)}{\pgfxy(10.46,-8.89)}
\pgfellipse[stroke]{\pgfxy(10.41,31.08)}{\pgfxy(1.86,-2.19)}{\pgfxy(10.46,8.89)}
\pgfellipse[stroke]{\pgfxy(31.41,51.08)}{\pgfxy(-1.86,-2.19)}{\pgfxy(10.46,-8.89)}
\pgfellipse[stroke]{\pgfxy(10.41,51.08)}{\pgfxy(1.86,-2.19)}{\pgfxy(10.46,8.89)}
\pgfellipse[stroke]{\pgfxy(30.41,61.08)}{\pgfxy(-1.86,-2.19)}{\pgfxy(10.46,-8.89)}
\pgfellipse[stroke]{\pgfxy(9.41,61.08)}{\pgfxy(1.86,-2.19)}{\pgfxy(10.46,8.89)}
\pgfputat{\pgfxy(23.00,22.00)}{\pgfbox[bottom,left]{\fontsize{7.97}{9.56}\selectfont \makebox[0pt]{$v$}}}
\pgfsetdash{{0.30mm}{0.50mm}}{0mm}\pgfsetlinewidth{0.30mm}\pgfmoveto{\pgfxy(10.00,45.00)}\pgflineto{\pgfxy(10.00,39.00)}\pgfstroke
\pgfmoveto{\pgfxy(30.00,45.00)}\pgflineto{\pgfxy(30.00,39.00)}\pgfstroke
\pgfputat{\pgfxy(27.92,61.99)}{\pgfbox[bottom,left]{\rotatebox{327.23}{\fontsize{7.97}{9.56}\selectfont \smash{\makebox[0pt][l]{$R_{\ell'}$}}}}}
\pgfputat{\pgfxy(28.98,52.22)}{\pgfbox[bottom,left]{\rotatebox{323.35}{\fontsize{7.97}{9.56}\selectfont \smash{\makebox[0pt][l]{$R_{\ell'-1}$}}}}}
\pgfputat{\pgfxy(34.00,26.00)}{\pgfbox[bottom,left]{\fontsize{7.97}{9.56}\selectfont $R_2$}}
\pgfputat{\pgfxy(33.00,16.00)}{\pgfbox[bottom,left]{\fontsize{7.97}{9.56}\selectfont $R_1$}}
\pgfsetdash{}{0mm}\pgfmoveto{\pgfxy(91.00,60.00)}\pgflineto{\pgfxy(90.00,70.00)}\pgfstroke
\pgfmoveto{\pgfxy(90.00,30.00)}\pgflineto{\pgfxy(91.00,40.00)}\pgfstroke
\pgfmoveto{\pgfxy(91.00,20.00)}\pgflineto{\pgfxy(90.00,30.00)}\pgfstroke
\pgfmoveto{\pgfxy(103.00,40.00)}\pgflineto{\pgfxy(103.00,40.00)}\pgfstroke
\pgfmoveto{\pgfxy(91.00,40.00)}\pgflineto{\pgfxy(91.00,50.00)}\pgfstroke
\pgfmoveto{\pgfxy(91.00,50.00)}\pgflineto{\pgfxy(91.00,60.00)}\pgfstroke
\pgfsetdash{{2.00mm}{1.00mm}}{0mm}\pgfsetlinewidth{0.15mm}\pgfellipse[stroke]{\pgfxy(110.41,31.08)}{\pgfxy(-1.86,-2.19)}{\pgfxy(10.46,-8.89)}
\pgfellipse[stroke]{\pgfxy(110.41,51.08)}{\pgfxy(-1.86,-2.19)}{\pgfxy(10.46,-8.89)}
\pgfsetdash{{0.30mm}{0.50mm}}{0mm}\pgfsetlinewidth{0.30mm}\pgfmoveto{\pgfxy(110.00,45.00)}\pgflineto{\pgfxy(110.00,39.00)}\pgfstroke
\pgfputat{\pgfxy(107.92,61.99)}{\pgfbox[bottom,left]{\rotatebox{327.23}{\fontsize{7.97}{9.56}\selectfont \smash{\makebox[0pt][l]{$R_{\ell'}$}}}}}
\pgfputat{\pgfxy(107.98,52.22)}{\pgfbox[bottom,left]{\rotatebox{323.35}{\fontsize{7.97}{9.56}\selectfont \smash{\makebox[0pt][l]{$R_{\ell'-1}$}}}}}
\pgfputat{\pgfxy(113.00,26.00)}{\pgfbox[bottom,left]{\fontsize{7.97}{9.56}\selectfont $R_2$}}
\pgfputat{\pgfxy(113.00,16.00)}{\pgfbox[bottom,left]{\fontsize{7.97}{9.56}\selectfont $R_1$}}
\pgfsetdash{{2.00mm}{1.00mm}}{0mm}\pgfsetlinewidth{0.15mm}\pgfellipse[stroke]{\pgfxy(110.41,61.08)}{\pgfxy(-1.86,-2.19)}{\pgfxy(10.46,-8.89)}
\pgfellipse[stroke]{\pgfxy(110.41,21.08)}{\pgfxy(-1.86,-2.19)}{\pgfxy(10.46,-8.89)}
\pgfellipse[stroke]{\pgfxy(79.41,21.08)}{\pgfxy(1.86,-2.19)}{\pgfxy(10.46,8.89)}
\pgfellipse[stroke]{\pgfxy(80.41,31.08)}{\pgfxy(1.86,-2.19)}{\pgfxy(10.46,8.89)}
\pgfellipse[stroke]{\pgfxy(80.41,51.08)}{\pgfxy(1.86,-2.19)}{\pgfxy(10.46,8.89)}
\pgfellipse[stroke]{\pgfxy(79.41,61.08)}{\pgfxy(1.86,-2.19)}{\pgfxy(10.46,8.89)}
\pgfputat{\pgfxy(93.00,22.00)}{\pgfbox[bottom,left]{\fontsize{7.97}{9.56}\selectfont \makebox[0pt]{$v$}}}
\pgfsetdash{{0.30mm}{0.50mm}}{0mm}\pgfsetlinewidth{0.30mm}\pgfmoveto{\pgfxy(80.00,45.00)}\pgflineto{\pgfxy(80.00,39.00)}\pgfstroke
\pgfputat{\pgfxy(80.00,72.00)}{\pgfbox[bottom,left]{\fontsize{7.97}{9.56}\selectfont \makebox[0pt]{$L$}}}
\pgfputat{\pgfxy(-12.00,9.00)}{\pgfbox[bottom,left]{\fontsize{7.97}{9.56}\selectfont \makebox[0pt][r]{level $\ell'+1$}}}
\pgfputat{\pgfxy(-12.00,19.00)}{\pgfbox[bottom,left]{\fontsize{7.97}{9.56}\selectfont \makebox[0pt][r]{level $\ell'$}}}
\pgfputat{\pgfxy(-12.00,29.00)}{\pgfbox[bottom,left]{\fontsize{7.97}{9.56}\selectfont \makebox[0pt][r]{level $\ell'-1$}}}
\pgfputat{\pgfxy(-12.00,59.00)}{\pgfbox[bottom,left]{\fontsize{7.97}{9.56}\selectfont \makebox[0pt][r]{level $1$}}}
\pgfputat{\pgfxy(-12.00,69.00)}{\pgfbox[bottom,left]{\fontsize{7.97}{9.56}\selectfont \makebox[0pt][r]{level $0$}}}
\pgfcircle[fill]{\pgfxy(20.00,70.00)}{0.70mm}
\pgfsetdash{}{0mm}\pgfcircle[stroke]{\pgfxy(20.00,70.00)}{0.70mm}
\pgfcircle[fill]{\pgfxy(90.00,70.00)}{0.70mm}
\pgfcircle[stroke]{\pgfxy(90.00,70.00)}{0.70mm}
\pgfputat{\pgfxy(20.00,72.00)}{\pgfbox[bottom,left]{\fontsize{7.97}{9.56}\selectfont \makebox[0pt]{$r$}}}
\pgfputat{\pgfxy(90.00,72.00)}{\pgfbox[bottom,left]{\fontsize{7.97}{9.56}\selectfont \makebox[0pt]{$r$}}}
\color[rgb]{1,0,1}\pgfmoveto{\pgfxy(24.00,18.00)}\pgfcurveto{\pgfxy(23.50,16.79)}{\pgfxy(22.31,16.00)}{\pgfxy(21.00,16.00)}\pgfcurveto{\pgfxy(19.69,16.00)}{\pgfxy(18.50,16.79)}{\pgfxy(18.00,18.00)}\pgfstroke
\color[rgb]{0,0,0}\pgfmoveto{\pgfxy(21.00,20.00)}\pgflineto{\pgfxy(18.00,10.00)}\pgfstroke
\pgfmoveto{\pgfxy(21.00,20.00)}\pgflineto{\pgfxy(23.00,10.00)}\pgfstroke
\pgfmoveto{\pgfxy(21.00,20.00)}\pgflineto{\pgfxy(27.00,10.00)}\pgfstroke
\pgfmoveto{\pgfxy(21.00,20.00)}\pgflineto{\pgfxy(14.00,10.00)}\pgfstroke
\pgfputat{\pgfxy(21.00,5.00)}{\pgfbox[bottom,left]{\fontsize{7.97}{9.56}\selectfont \makebox[0pt]{$D_v$}}}
\pgfputat{\pgfxy(17.00,14.00)}{\pgfbox[bottom,left]{\fontsize{7.97}{9.56}\selectfont \makebox[0pt][r]{$k'$}}}
\pgfcircle[fill]{\pgfxy(21.00,20.00)}{0.70mm}
\pgfcircle[stroke]{\pgfxy(21.00,20.00)}{0.70mm}
\pgfsetdash{{2.00mm}{1.00mm}}{0mm}\pgfsetlinewidth{0.15mm}\pgfellipse[stroke]{\pgfxy(20.50,10.00)}{\pgfxy(7.50,0.00)}{\pgfxy(0.00,10.00)}
\pgfcircle[fill]{\pgfxy(91.00,20.00)}{0.70mm}
\pgfsetdash{}{0mm}\pgfsetlinewidth{0.30mm}\pgfcircle[stroke]{\pgfxy(91.00,20.00)}{0.70mm}
\color[rgb]{1,0,1}\pgfmoveto{\pgfxy(106.00,18.00)}\pgfcurveto{\pgfxy(105.50,16.79)}{\pgfxy(104.31,16.00)}{\pgfxy(103.00,16.00)}\pgfcurveto{\pgfxy(101.69,16.00)}{\pgfxy(100.50,16.79)}{\pgfxy(100.00,18.00)}\pgfstroke
\color[rgb]{0,0,0}\pgfmoveto{\pgfxy(103.00,20.00)}\pgflineto{\pgfxy(100.00,10.00)}\pgfstroke
\pgfmoveto{\pgfxy(103.00,20.00)}\pgflineto{\pgfxy(105.00,10.00)}\pgfstroke
\pgfmoveto{\pgfxy(103.00,20.00)}\pgflineto{\pgfxy(109.00,10.00)}\pgfstroke
\pgfmoveto{\pgfxy(103.00,20.00)}\pgflineto{\pgfxy(96.00,10.00)}\pgfstroke
\pgfputat{\pgfxy(103.00,5.00)}{\pgfbox[bottom,left]{\fontsize{7.97}{9.56}\selectfont \makebox[0pt]{$D_v$}}}
\pgfputat{\pgfxy(99.00,14.00)}{\pgfbox[bottom,left]{\fontsize{7.97}{9.56}\selectfont \makebox[0pt][r]{$k'$}}}
\pgfsetdash{{2.00mm}{1.00mm}}{0mm}\pgfsetlinewidth{0.15mm}\pgfellipse[stroke]{\pgfxy(102.50,10.00)}{\pgfxy(7.50,0.00)}{\pgfxy(0.00,10.00)}
\end{pgfpicture}%
\caption{Tree decomposition}
\label{fig:decomposition}
\end{figure}
Clearly, the outdegree of the root of $D_v$ is $k'$.
In preorder, the vertex $v$ is the last leaf at level $\ell'$ in the tree $L$.

From $(T,v)$, we define a lattice path $P$ of length $(2n+\ell'+1)$ from $(0,0)$ to $(2n+\ell'+1, -\ell'-1)$ by
$$P = \psi(D_v) \searrow \varphi(R_1) \searrow \varphi(R_2) \searrow \dots \searrow \varphi(R_{\ell'}) \searrow \varphi(L),$$
where $\searrow$ means a down-step.
By convention, if $\ell'=0$.
We set $$P = \psi(D_v) \searrow \varphi(\emptyset).$$
Especially, if $\ell'>0$, the lattice path $P$ always starts with $k'$ consecutive up-steps and ends with one up-step and $\ell'$ consecutive down-steps as the Figure~\ref{fig:mainlemma}.
\begin{figure}[t]
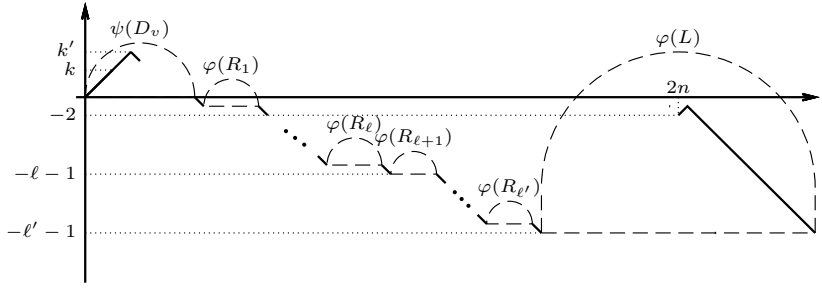

\centering
\begin{pgfpicture}{-17.87mm}{-26.60mm}{86.00mm}{14.00mm}
\pgfsetxvec{\pgfpoint{0.60mm}{0mm}}
\pgfsetyvec{\pgfpoint{0mm}{0.60mm}}
\color[rgb]{0,0,0}\pgfsetlinewidth{0.30mm}\pgfsetdash{}{0mm}
\pgfsetdash{{0.15mm}{0.50mm}}{0mm}\pgfsetlinewidth{0.15mm}\pgfmoveto{\pgfxy(-20.00,-17.00)}\pgflineto{\pgfxy(47.00,-17.00)}\pgfstroke
\pgfmoveto{\pgfxy(-20.00,-4.00)}\pgflineto{\pgfxy(110.00,-4.00)}\pgfstroke
\pgfmoveto{\pgfxy(-20.00,-30.00)}\pgflineto{\pgfxy(80.00,-30.00)}\pgfstroke
\pgfsetdash{}{0mm}\pgfsetlinewidth{0.30mm}\pgfmoveto{\pgfxy(-22.00,0.00)}\pgflineto{\pgfxy(140.00,0.00)}\pgfstroke
\pgfmoveto{\pgfxy(140.00,0.00)}\pgflineto{\pgfxy(137.20,0.70)}\pgflineto{\pgfxy(140.00,0.00)}\pgflineto{\pgfxy(137.20,-0.70)}\pgflineto{\pgfxy(140.00,0.00)}\pgfclosepath\pgffill
\pgfmoveto{\pgfxy(140.00,0.00)}\pgflineto{\pgfxy(137.20,0.70)}\pgflineto{\pgfxy(140.00,0.00)}\pgflineto{\pgfxy(137.20,-0.70)}\pgflineto{\pgfxy(140.00,0.00)}\pgfclosepath\pgfstroke
\pgfmoveto{\pgfxy(-20.00,-41.00)}\pgflineto{\pgfxy(-20.00,20.00)}\pgfstroke
\pgfmoveto{\pgfxy(-20.00,20.00)}\pgflineto{\pgfxy(-20.70,17.20)}\pgflineto{\pgfxy(-20.00,20.00)}\pgflineto{\pgfxy(-19.30,17.20)}\pgflineto{\pgfxy(-20.00,20.00)}\pgfclosepath\pgffill
\pgfmoveto{\pgfxy(-20.00,20.00)}\pgflineto{\pgfxy(-20.70,17.20)}\pgflineto{\pgfxy(-20.00,20.00)}\pgflineto{\pgfxy(-19.30,17.20)}\pgflineto{\pgfxy(-20.00,20.00)}\pgfclosepath\pgfstroke
\pgfmoveto{\pgfxy(4.00,0.00)}\pgflineto{\pgfxy(6.00,-2.00)}\pgfstroke
\pgfmoveto{\pgfxy(18.00,-2.00)}\pgflineto{\pgfxy(20.00,-4.00)}\pgfstroke
\pgfmoveto{\pgfxy(66.00,-26.00)}\pgflineto{\pgfxy(68.00,-28.00)}\pgfstroke
\pgfmoveto{\pgfxy(78.00,-28.00)}\pgflineto{\pgfxy(80.00,-30.00)}\pgfstroke
\pgfmoveto{\pgfxy(-7.00,4.00)}\pgfclosepath\pgfstroke
\pgfsetdash{{2.00mm}{1.00mm}}{0mm}\pgfsetlinewidth{0.15mm}\pgfmoveto{\pgfxy(18.00,-2.00)}\pgfcurveto{\pgfxy(18.00,-0.41)}{\pgfxy(17.37,1.12)}{\pgfxy(16.24,2.24)}\pgfcurveto{\pgfxy(15.12,3.37)}{\pgfxy(13.59,4.00)}{\pgfxy(12.00,4.00)}\pgfcurveto{\pgfxy(10.41,4.00)}{\pgfxy(8.88,3.37)}{\pgfxy(7.76,2.24)}\pgfcurveto{\pgfxy(6.63,1.12)}{\pgfxy(6.00,-0.41)}{\pgfxy(6.00,-2.00)}\pgfclosepath\pgfstroke
\pgfmoveto{\pgfxy(78.00,-28.00)}\pgfcurveto{\pgfxy(78.00,-26.67)}{\pgfxy(77.47,-25.40)}{\pgfxy(76.54,-24.46)}\pgfcurveto{\pgfxy(75.60,-23.53)}{\pgfxy(74.33,-23.00)}{\pgfxy(73.00,-23.00)}\pgfcurveto{\pgfxy(71.67,-23.00)}{\pgfxy(70.40,-23.53)}{\pgfxy(69.46,-24.46)}\pgfcurveto{\pgfxy(68.53,-25.40)}{\pgfxy(68.00,-26.67)}{\pgfxy(68.00,-28.00)}\pgfclosepath\pgfstroke
\pgfputat{\pgfxy(-22.00,-31.00)}{\pgfbox[bottom,left]{\fontsize{6.83}{8.19}\selectfont \makebox[0pt][r]{$-\ell'-1$}}}
\pgfputat{\pgfxy(-8.00,14.00)}{\pgfbox[bottom,left]{\fontsize{6.83}{8.19}\selectfont \makebox[0pt]{$\psi(D_v)$}}}
\pgfputat{\pgfxy(12.00,6.00)}{\pgfbox[bottom,left]{\fontsize{6.83}{8.19}\selectfont \makebox[0pt]{$\varphi(R_1)$}}}
\pgfputat{\pgfxy(73.00,-21.00)}{\pgfbox[bottom,left]{\fontsize{6.83}{8.19}\selectfont \makebox[0pt]{$\varphi(R_{\ell'})$}}}
\pgfsetdash{}{0mm}\pgfsetlinewidth{0.30mm}\pgfmoveto{\pgfxy(-20.00,0.00)}\pgflineto{\pgfxy(-10.00,10.00)}\pgfstroke
\pgfsetdash{{2.00mm}{1.00mm}}{0mm}\pgfsetlinewidth{0.15mm}\pgfmoveto{\pgfxy(4.00,0.00)}\pgfcurveto{\pgfxy(4.00,3.18)}{\pgfxy(2.74,6.23)}{\pgfxy(0.49,8.49)}\pgfcurveto{\pgfxy(-1.77,10.74)}{\pgfxy(-4.82,12.00)}{\pgfxy(-8.00,12.00)}\pgfcurveto{\pgfxy(-11.18,12.00)}{\pgfxy(-14.23,10.74)}{\pgfxy(-16.49,8.49)}\pgfcurveto{\pgfxy(-18.74,6.23)}{\pgfxy(-20.00,3.18)}{\pgfxy(-20.00,0.00)}\pgfcurveto{\pgfxy(-17.00,0.00)}{\pgfxy(-11.00,0.00)}{\pgfxy(-8.00,0.00)}\pgfclosepath\pgfstroke
\pgfsetdash{}{0mm}\pgfsetlinewidth{0.30mm}\pgfmoveto{\pgfxy(-20.00,0.00)}\pgflineto{\pgfxy(-14.00,6.00)}\pgfstroke
\pgfsetdash{{2.00mm}{1.00mm}}{0mm}\pgfsetlinewidth{0.15mm}\pgfmoveto{\pgfxy(45.00,-15.00)}\pgfcurveto{\pgfxy(45.00,-13.41)}{\pgfxy(44.37,-11.88)}{\pgfxy(43.24,-10.76)}\pgfcurveto{\pgfxy(42.12,-9.63)}{\pgfxy(40.59,-9.00)}{\pgfxy(39.00,-9.00)}\pgfcurveto{\pgfxy(37.41,-9.00)}{\pgfxy(35.88,-9.63)}{\pgfxy(34.76,-10.76)}\pgfcurveto{\pgfxy(33.63,-11.88)}{\pgfxy(33.00,-13.41)}{\pgfxy(33.00,-15.00)}\pgfcurveto{\pgfxy(34.50,-15.00)}{\pgfxy(37.50,-15.00)}{\pgfxy(39.00,-15.00)}\pgfclosepath\pgfstroke
\pgfsetdash{}{0mm}\pgfsetlinewidth{0.30mm}\pgfmoveto{\pgfxy(31.00,-13.00)}\pgflineto{\pgfxy(33.00,-15.00)}\pgfstroke
\pgfputat{\pgfxy(39.00,-7.00)}{\pgfbox[bottom,left]{\fontsize{6.83}{8.19}\selectfont \makebox[0pt]{$\varphi(R_{\ell})$}}}
\pgfputat{\pgfxy(-22.00,5.00)}{\pgfbox[bottom,left]{\fontsize{6.83}{8.19}\selectfont \makebox[0pt][r]{$k$}}}
\pgfsetdash{{0.15mm}{0.50mm}}{0mm}\pgfsetlinewidth{0.15mm}\pgfmoveto{\pgfxy(-20.00,6.00)}\pgflineto{\pgfxy(-14.00,6.00)}\pgfstroke
\pgfmoveto{\pgfxy(-20.00,10.00)}\pgflineto{\pgfxy(-10.00,10.00)}\pgfstroke
\pgfputat{\pgfxy(-22.00,9.00)}{\pgfbox[bottom,left]{\fontsize{6.83}{8.19}\selectfont \makebox[0pt][r]{$k'$}}}
\pgfsetdash{}{0mm}\pgfsetlinewidth{0.30mm}\pgfmoveto{\pgfxy(-10.00,10.00)}\pgflineto{\pgfxy(-8.00,8.00)}\pgfstroke
\pgfmoveto{\pgfxy(118.00,-30.00)}\pgflineto{\pgfxy(118.00,-30.00)}\pgfstroke
\pgfmoveto{\pgfxy(110.00,-4.00)}\pgflineto{\pgfxy(112.00,-2.00)}\pgflineto{\pgfxy(140.00,-30.00)}\pgfstroke
\pgfsetdash{{2.00mm}{1.00mm}}{0mm}\pgfsetlinewidth{0.15mm}\pgfmoveto{\pgfxy(140.00,-20.00)}\pgfcurveto{\pgfxy(140.00,-12.04)}{\pgfxy(136.84,-4.41)}{\pgfxy(131.21,1.21)}\pgfcurveto{\pgfxy(125.59,6.84)}{\pgfxy(117.96,10.00)}{\pgfxy(110.00,10.00)}\pgfcurveto{\pgfxy(102.04,10.00)}{\pgfxy(94.41,6.84)}{\pgfxy(88.79,1.21)}\pgfcurveto{\pgfxy(83.16,-4.41)}{\pgfxy(80.00,-12.04)}{\pgfxy(80.00,-20.00)}\pgfstroke
\pgfmoveto{\pgfxy(80.00,-30.00)}\pgflineto{\pgfxy(80.00,-20.00)}\pgfstroke
\pgfmoveto{\pgfxy(140.00,-30.00)}\pgflineto{\pgfxy(140.00,-20.00)}\pgfstroke
\pgfputat{\pgfxy(110.00,1.00)}{\pgfbox[bottom,left]{\fontsize{6.83}{8.19}\selectfont \makebox[0pt]{$2n$}}}
\pgfmoveto{\pgfxy(80.00,-30.00)}\pgflineto{\pgfxy(140.00,-30.00)}\pgfstroke
\pgfputat{\pgfxy(110.00,12.00)}{\pgfbox[bottom,left]{\fontsize{6.83}{8.19}\selectfont \makebox[0pt]{$\varphi(L)$}}}
\pgfsetdash{}{0mm}\pgfsetlinewidth{0.30mm}\pgfmoveto{\pgfxy(108.00,-2.00)}\pgflineto{\pgfxy(108.00,-2.00)}\pgfstroke
\pgfputat{\pgfxy(-22.00,-5.00)}{\pgfbox[bottom,left]{\fontsize{6.83}{8.19}\selectfont \makebox[0pt][r]{$-2$}}}
\pgfmoveto{\pgfxy(45.00,-15.00)}\pgflineto{\pgfxy(47.00,-17.00)}\pgfstroke
\pgfsetdash{{2.00mm}{1.00mm}}{0mm}\pgfsetlinewidth{0.15mm}\pgfmoveto{\pgfxy(57.00,-17.00)}\pgfcurveto{\pgfxy(57.00,-15.67)}{\pgfxy(56.47,-14.40)}{\pgfxy(55.54,-13.46)}\pgfcurveto{\pgfxy(54.60,-12.53)}{\pgfxy(53.33,-12.00)}{\pgfxy(52.00,-12.00)}\pgfcurveto{\pgfxy(50.67,-12.00)}{\pgfxy(49.40,-12.53)}{\pgfxy(48.46,-13.46)}\pgfcurveto{\pgfxy(47.53,-14.40)}{\pgfxy(47.00,-15.67)}{\pgfxy(47.00,-17.00)}\pgfclosepath\pgfstroke
\pgfputat{\pgfxy(52.00,-10.00)}{\pgfbox[bottom,left]{\fontsize{6.83}{8.19}\selectfont \makebox[0pt]{$\varphi(R_{\ell+1})$}}}
\pgfputat{\pgfxy(-22.00,-18.00)}{\pgfbox[bottom,left]{\fontsize{6.83}{8.19}\selectfont \makebox[0pt][r]{$-\ell-1$}}}
\pgfsetdash{}{0mm}\pgfsetlinewidth{0.30mm}\pgfmoveto{\pgfxy(57.00,-17.00)}\pgflineto{\pgfxy(59.00,-19.00)}\pgfstroke
\pgfsetdash{{0.15mm}{0.50mm}}{0mm}\pgfsetlinewidth{0.15mm}\pgfmoveto{\pgfxy(110.00,0.00)}\pgflineto{\pgfxy(110.00,-4.00)}\pgfstroke
\pgfcircle[fill]{\pgfxy(61.00,-21.00)}{0.12mm}
\pgfsetdash{}{0mm}\pgfsetlinewidth{0.30mm}\pgfcircle[stroke]{\pgfxy(61.00,-21.00)}{0.12mm}
\pgfcircle[fill]{\pgfxy(62.50,-22.50)}{0.12mm}
\pgfcircle[stroke]{\pgfxy(62.50,-22.50)}{0.12mm}
\pgfcircle[fill]{\pgfxy(64.00,-24.00)}{0.12mm}
\pgfcircle[stroke]{\pgfxy(64.00,-24.00)}{0.12mm}
\pgfcircle[fill]{\pgfxy(24.00,-7.50)}{0.12mm}
\pgfcircle[stroke]{\pgfxy(24.00,-7.50)}{0.12mm}
\pgfcircle[fill]{\pgfxy(26.00,-9.00)}{0.12mm}
\pgfcircle[stroke]{\pgfxy(26.00,-9.00)}{0.12mm}
\pgfcircle[fill]{\pgfxy(28.00,-10.50)}{0.12mm}
\pgfcircle[stroke]{\pgfxy(28.00,-10.50)}{0.12mm}
\end{pgfpicture}%
\caption{Outline of a lattice path $P$ induced from tree decomposition}
\label{fig:mainlemma}
\end{figure}

When we draw the line $y=-\ell-1$,
because the $y$-coordinate of the lowest point of $P$ is $-\ell'-1$,
this line should meet the lattice path $P$.
Denote $p$ by the first meeting point of $P$ and $y=-\ell-1$.
Replacing the portion of $P$ from $p$ to $(2n+\ell'+1, -\ell'-1)$
by its reflection about $y=-\ell-1$,
we obtain the lattice path $\tilde{P}$ from $(0, 0)$ to $(2n+\ell'+1, -2\ell+\ell'-1)$,
which also always starts with $k'$ consecutive up-steps
and ends with one down-step and $\ell'$ consecutive up-steps
as the Figure~\ref{fig:mainlemma1}, even if $\ell'=0$.
\begin{figure}[t]
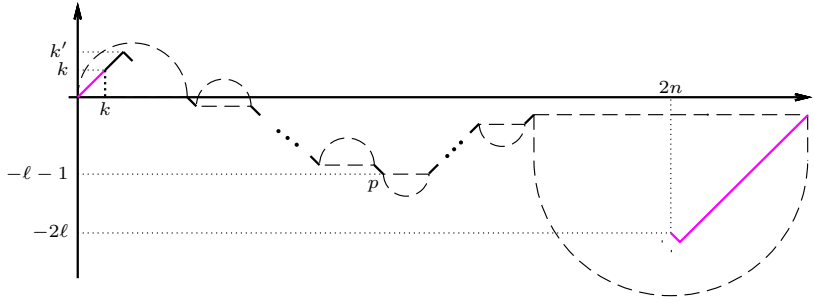

\centering
\begin{pgfpicture}{-19.07mm}{-28.32mm}{86.00mm}{14.00mm}
\pgfsetxvec{\pgfpoint{0.60mm}{0mm}}
\pgfsetyvec{\pgfpoint{0mm}{0.60mm}}
\color[rgb]{0,0,0}\pgfsetlinewidth{0.30mm}\pgfsetdash{}{0mm}
\pgfsetdash{{0.15mm}{0.50mm}}{0mm}\pgfsetlinewidth{0.15mm}\pgfmoveto{\pgfxy(-20.00,-17.00)}\pgflineto{\pgfxy(47.00,-17.00)}\pgfstroke
\pgfmoveto{\pgfxy(-20.00,-30.00)}\pgflineto{\pgfxy(110.00,-30.00)}\pgfstroke
\pgfsetdash{}{0mm}\pgfsetlinewidth{0.30mm}\pgfmoveto{\pgfxy(-22.00,0.00)}\pgflineto{\pgfxy(140.00,0.00)}\pgfstroke
\pgfmoveto{\pgfxy(140.00,0.00)}\pgflineto{\pgfxy(137.20,0.70)}\pgflineto{\pgfxy(140.00,0.00)}\pgflineto{\pgfxy(137.20,-0.70)}\pgflineto{\pgfxy(140.00,0.00)}\pgfclosepath\pgffill
\pgfmoveto{\pgfxy(140.00,0.00)}\pgflineto{\pgfxy(137.20,0.70)}\pgflineto{\pgfxy(140.00,0.00)}\pgflineto{\pgfxy(137.20,-0.70)}\pgflineto{\pgfxy(140.00,0.00)}\pgfclosepath\pgfstroke
\pgfmoveto{\pgfxy(-20.00,-40.00)}\pgflineto{\pgfxy(-20.00,20.00)}\pgfstroke
\pgfmoveto{\pgfxy(-20.00,20.00)}\pgflineto{\pgfxy(-20.70,17.20)}\pgflineto{\pgfxy(-20.00,20.00)}\pgflineto{\pgfxy(-19.30,17.20)}\pgflineto{\pgfxy(-20.00,20.00)}\pgfclosepath\pgffill
\pgfmoveto{\pgfxy(-20.00,20.00)}\pgflineto{\pgfxy(-20.70,17.20)}\pgflineto{\pgfxy(-20.00,20.00)}\pgflineto{\pgfxy(-19.30,17.20)}\pgflineto{\pgfxy(-20.00,20.00)}\pgfclosepath\pgfstroke
\pgfmoveto{\pgfxy(4.00,0.00)}\pgflineto{\pgfxy(6.00,-2.00)}\pgfstroke
\pgfmoveto{\pgfxy(18.00,-2.00)}\pgflineto{\pgfxy(20.00,-4.00)}\pgfstroke
\pgfmoveto{\pgfxy(66.00,-7.86)}\pgflineto{\pgfxy(68.00,-5.86)}\pgfstroke
\pgfmoveto{\pgfxy(78.00,-5.86)}\pgflineto{\pgfxy(80.00,-3.86)}\pgfstroke
\pgfmoveto{\pgfxy(-7.00,4.00)}\pgfclosepath\pgfstroke
\pgfsetdash{{2.00mm}{1.00mm}}{0mm}\pgfsetlinewidth{0.15mm}\pgfmoveto{\pgfxy(18.00,-2.00)}\pgfcurveto{\pgfxy(18.00,-0.41)}{\pgfxy(17.37,1.12)}{\pgfxy(16.24,2.24)}\pgfcurveto{\pgfxy(15.12,3.37)}{\pgfxy(13.59,4.00)}{\pgfxy(12.00,4.00)}\pgfcurveto{\pgfxy(10.41,4.00)}{\pgfxy(8.88,3.37)}{\pgfxy(7.76,2.24)}\pgfcurveto{\pgfxy(6.63,1.12)}{\pgfxy(6.00,-0.41)}{\pgfxy(6.00,-2.00)}\pgfclosepath\pgfstroke
\pgfmoveto{\pgfxy(68.00,-6.00)}\pgfcurveto{\pgfxy(68.04,-7.30)}{\pgfxy(68.58,-8.54)}{\pgfxy(69.51,-9.45)}\pgfcurveto{\pgfxy(70.45,-10.36)}{\pgfxy(71.70,-10.87)}{\pgfxy(73.00,-10.87)}\pgfcurveto{\pgfxy(74.30,-10.87)}{\pgfxy(75.55,-10.36)}{\pgfxy(76.49,-9.45)}\pgfcurveto{\pgfxy(77.42,-8.54)}{\pgfxy(77.96,-7.30)}{\pgfxy(78.00,-6.00)}\pgfclosepath\pgfstroke
\pgfsetdash{}{0mm}\pgfsetlinewidth{0.30mm}\pgfmoveto{\pgfxy(-14.00,6.00)}\pgflineto{\pgfxy(-10.00,10.00)}\pgfstroke
\pgfsetdash{{2.00mm}{1.00mm}}{0mm}\pgfsetlinewidth{0.15mm}\pgfmoveto{\pgfxy(4.00,0.00)}\pgfcurveto{\pgfxy(4.00,3.18)}{\pgfxy(2.74,6.23)}{\pgfxy(0.49,8.49)}\pgfcurveto{\pgfxy(-1.77,10.74)}{\pgfxy(-4.82,12.00)}{\pgfxy(-8.00,12.00)}\pgfcurveto{\pgfxy(-11.18,12.00)}{\pgfxy(-14.23,10.74)}{\pgfxy(-16.49,8.49)}\pgfcurveto{\pgfxy(-18.74,6.23)}{\pgfxy(-20.00,3.18)}{\pgfxy(-20.00,0.00)}\pgfcurveto{\pgfxy(-17.00,0.00)}{\pgfxy(-11.00,0.00)}{\pgfxy(-8.00,0.00)}\pgfclosepath\pgfstroke
\pgfmoveto{\pgfxy(45.00,-15.00)}\pgfcurveto{\pgfxy(45.00,-13.41)}{\pgfxy(44.37,-11.88)}{\pgfxy(43.24,-10.76)}\pgfcurveto{\pgfxy(42.12,-9.63)}{\pgfxy(40.59,-9.00)}{\pgfxy(39.00,-9.00)}\pgfcurveto{\pgfxy(37.41,-9.00)}{\pgfxy(35.88,-9.63)}{\pgfxy(34.76,-10.76)}\pgfcurveto{\pgfxy(33.63,-11.88)}{\pgfxy(33.00,-13.41)}{\pgfxy(33.00,-15.00)}\pgfcurveto{\pgfxy(34.50,-15.00)}{\pgfxy(37.50,-15.00)}{\pgfxy(39.00,-15.00)}\pgfclosepath\pgfstroke
\pgfsetdash{}{0mm}\pgfsetlinewidth{0.30mm}\pgfmoveto{\pgfxy(31.00,-13.00)}\pgflineto{\pgfxy(33.00,-15.00)}\pgfstroke
\pgfputat{\pgfxy(-22.00,5.00)}{\pgfbox[bottom,left]{\fontsize{6.83}{8.19}\selectfont \makebox[0pt][r]{$k$}}}
\pgfsetdash{{0.15mm}{0.50mm}}{0mm}\pgfsetlinewidth{0.15mm}\pgfmoveto{\pgfxy(-20.00,6.00)}\pgflineto{\pgfxy(-14.00,6.00)}\pgfstroke
\pgfmoveto{\pgfxy(-20.00,10.00)}\pgflineto{\pgfxy(-10.00,10.00)}\pgfstroke
\pgfputat{\pgfxy(-22.00,9.00)}{\pgfbox[bottom,left]{\fontsize{6.83}{8.19}\selectfont \makebox[0pt][r]{$k'$}}}
\pgfsetdash{}{0mm}\pgfsetlinewidth{0.30mm}\pgfmoveto{\pgfxy(-10.00,10.00)}\pgflineto{\pgfxy(-8.00,8.00)}\pgfstroke
\pgfmoveto{\pgfxy(118.00,-3.86)}\pgflineto{\pgfxy(118.00,-3.86)}\pgfstroke
\pgfsetdash{{2.00mm}{1.00mm}}{0mm}\pgfsetlinewidth{0.15mm}\pgfmoveto{\pgfxy(80.00,-14.00)}\pgfcurveto{\pgfxy(80.04,-21.93)}{\pgfxy(83.21,-29.53)}{\pgfxy(88.83,-35.13)}\pgfcurveto{\pgfxy(94.46,-40.72)}{\pgfxy(102.07,-43.86)}{\pgfxy(110.00,-43.86)}\pgfcurveto{\pgfxy(117.93,-43.86)}{\pgfxy(125.54,-40.72)}{\pgfxy(131.17,-35.13)}\pgfcurveto{\pgfxy(136.79,-29.53)}{\pgfxy(139.96,-21.93)}{\pgfxy(140.00,-14.00)}\pgfstroke
\pgfmoveto{\pgfxy(80.00,-3.86)}\pgflineto{\pgfxy(80.00,-13.86)}\pgfstroke
\pgfmoveto{\pgfxy(140.00,-3.86)}\pgflineto{\pgfxy(140.00,-14.00)}\pgfstroke
\pgfputat{\pgfxy(110.00,1.00)}{\pgfbox[bottom,left]{\fontsize{6.83}{8.19}\selectfont \makebox[0pt]{$2n$}}}
\pgfmoveto{\pgfxy(80.00,-3.86)}\pgflineto{\pgfxy(140.00,-3.86)}\pgfstroke
\pgfsetdash{}{0mm}\pgfsetlinewidth{0.30mm}\pgfmoveto{\pgfxy(108.00,-31.86)}\pgflineto{\pgfxy(108.00,-31.86)}\pgfstroke
\pgfputat{\pgfxy(-22.00,-31.00)}{\pgfbox[bottom,left]{\fontsize{6.83}{8.19}\selectfont \makebox[0pt][r]{$-2\ell$}}}
\pgfmoveto{\pgfxy(45.00,-15.00)}\pgflineto{\pgfxy(47.00,-17.00)}\pgfstroke
\pgfsetdash{{2.00mm}{1.00mm}}{0mm}\pgfsetlinewidth{0.15mm}\pgfmoveto{\pgfxy(47.00,-17.00)}\pgfcurveto{\pgfxy(47.04,-18.30)}{\pgfxy(47.58,-19.54)}{\pgfxy(48.51,-20.45)}\pgfcurveto{\pgfxy(49.45,-21.36)}{\pgfxy(50.70,-21.87)}{\pgfxy(52.00,-21.87)}\pgfcurveto{\pgfxy(53.30,-21.87)}{\pgfxy(54.55,-21.36)}{\pgfxy(55.49,-20.45)}\pgfcurveto{\pgfxy(56.42,-19.54)}{\pgfxy(56.96,-18.30)}{\pgfxy(57.00,-17.00)}\pgfclosepath\pgfstroke
\pgfputat{\pgfxy(-22.00,-18.00)}{\pgfbox[bottom,left]{\fontsize{6.83}{8.19}\selectfont \makebox[0pt][r]{$-\ell-1$}}}
\pgfsetdash{}{0mm}\pgfsetlinewidth{0.30mm}\pgfmoveto{\pgfxy(57.00,-16.86)}\pgflineto{\pgfxy(59.00,-14.86)}\pgfstroke
\pgfputat{\pgfxy(46.00,-20.00)}{\pgfbox[bottom,left]{\fontsize{6.83}{8.19}\selectfont \makebox[0pt][r]{$p$}}}
\pgfsetdash{{0.15mm}{0.50mm}}{0mm}\pgfsetlinewidth{0.15mm}\pgfmoveto{\pgfxy(110.00,-30.00)}\pgflineto{\pgfxy(110.00,0.00)}\pgfstroke
\color[rgb]{1,0,1}\pgfsetdash{}{0mm}\pgfsetlinewidth{0.30mm}\pgfmoveto{\pgfxy(-20.00,0.00)}\pgflineto{\pgfxy(-14.00,6.00)}\pgfstroke
\color[rgb]{0,0,0}\pgfmoveto{\pgfxy(110.00,-34.00)}\pgflineto{\pgfxy(110.00,-34.00)}\pgfstroke
\color[rgb]{1,0,1}\pgfmoveto{\pgfxy(110.00,-30.00)}\pgflineto{\pgfxy(112.00,-32.00)}\pgflineto{\pgfxy(140.00,-4.00)}\pgfstroke
\color[rgb]{0,0,0}\pgfsetdash{{0.30mm}{0.50mm}}{0mm}\pgfmoveto{\pgfxy(-14.00,6.00)}\pgflineto{\pgfxy(-14.00,0.00)}\pgfstroke
\pgfputat{\pgfxy(-14.00,-4.00)}{\pgfbox[bottom,left]{\fontsize{6.83}{8.19}\selectfont \makebox[0pt]{$k$}}}
\pgfcircle[fill]{\pgfxy(24.00,-7.50)}{0.12mm}
\pgfsetdash{}{0mm}\pgfcircle[stroke]{\pgfxy(24.00,-7.50)}{0.12mm}
\pgfcircle[fill]{\pgfxy(26.00,-9.00)}{0.12mm}
\pgfcircle[stroke]{\pgfxy(26.00,-9.00)}{0.12mm}
\pgfcircle[fill]{\pgfxy(28.00,-10.50)}{0.12mm}
\pgfcircle[stroke]{\pgfxy(28.00,-10.50)}{0.12mm}
\pgfcircle[fill]{\pgfxy(61.00,-13.00)}{0.12mm}
\pgfcircle[stroke]{\pgfxy(61.00,-13.00)}{0.12mm}
\pgfcircle[fill]{\pgfxy(62.50,-11.50)}{0.12mm}
\pgfcircle[stroke]{\pgfxy(62.50,-11.50)}{0.12mm}
\pgfcircle[fill]{\pgfxy(64.00,-10.00)}{0.12mm}
\pgfcircle[stroke]{\pgfxy(64.00,-10.00)}{0.12mm}
\end{pgfpicture}%
\caption{Outline of a lattice path $\tilde{P}$ and $\hat{P}$ by reflection about $y=-\ell-1$}
\label{fig:mainlemma1}
\end{figure}
Note that, if $\ell'=0$, $p$ should be the last point $(2n+1, -1)$ of $P$ and ${\tilde P}=P$.

By removing the first $k$ steps and the last $(\ell'+1)$ steps from $\tilde{P}$, we obtain the lattice path $\hat{P}$ from $(k, k)$ to $(2n, -2\ell)$.
Thus $\hat{P}$ belongs to $\mathcal{L}$ and define $\Phi(T,v) = \hat{P}$.
Since the reflection and the removal are reversible,
$\Phi$ is a bijection.
\end{proof}

\begin{thm}
Given $n \ge 1$,
for any two nonnegative integers $k$ and $\ell$,
the number of vertices of outdegree $k$ at level $\ell$ among trees in $\T_n$ is
\begin{align}
\label{eq:enumer}
\frac{k+2\ell}{2n-k} \binom{2n -k }{ n +\ell}.
\end{align}
\end{thm}

\begin{proof}
Using a sieve method, we obtain \eqref{eq:enumer} from \eqref{eq:lem} by
\begin{align*}
  \binom{2n -k   }{ n +\ell  }
- \binom{2n -k-1 }{ n +\ell  }
- \binom{2n -k   }{ n +\ell+1}
+ \binom{2n -k-1 }{ n +\ell+1}.
\end{align*}
\end{proof}

Thus, obtaining \eqref{eq:main} from \eqref{eq:enumer} by changing the index $\ell$,
we may count the set $\mathcal{D}$ as \eqref{item:c4} of Theorem~\ref{thm:main}.

\section*{Acknowledgements}

For the first author, this research was partially supported by Ministry of Science of Technology, Taiwan ROC under grants 104-2115-M-003-014-MY3.
For the second author, this study has been worked with the support of a research grant of Kangwon National University in 2015.
For the third author, this work was supported by Inha University Research Grant.



\end{document}